\documentclass[11pt,reqno]{amsproc}
\usepackage[margin=1in]{geometry}
\usepackage{amsmath, amsthm, amssymb}
\usepackage[colorlinks=true, pdfstartview=FitV, linkcolor=blue,citecolor=blue, urlcolor=blue]{hyperref}
\usepackage[abbrev,lite,nobysame]{amsrefs}
\usepackage{times}
\usepackage[usenames,dvipsnames]{color}
\usepackage{mathtools}

\mathtoolsset{showonlyrefs=true}

\def\eps{\varepsilon}
\def\e{{\rm e}}
\def\dist{{\rm dist}}

\def\dd{{\rm d}}
\def\ddt{{\frac{\dd}{\dd t}}}

\def\R {\mathbb{R}}
\def\u {\boldsymbol{u}}

\def\w {\boldsymbol{w}}

\def\N {\mathbb{N}}

\def\RR {{\mathcal R}}

\def\Q {{\mathcal Q}}

\def\ZZ {{\mathbb Z}}

\def\TT {{\mathbb T}^2}

\DeclareMathOperator*{\esup}{ess\,sup}
\def\de{{\partial}}

\newtheorem{proposition}{Proposition}[section]
\newtheorem{theorem}[proposition]{Theorem}
\newtheorem{corollary}[proposition]{Corollary}
\newtheorem{lemma}[proposition]{Lemma}
\theoremstyle{definition}

\newtheorem{remark}[proposition]{Remark}

\numberwithin{equation}{section}

\title[Subcritical SQG equations in scale-invariant spaces]{Long time behavior and critical limit of subcritical SQG equations in scale-invariant Sobolev spaces}

\author[M. Coti Zelati]{Michele Coti Zelati}

\address{Department of Mathematics, University of Maryland, College Park, MD 20742, USA}
\email{micotize@umd.edu}



\begin{document}

\begin{abstract}
We consider the subcritical SQG equation in its natural scale invariant Sobolev space and prove the
existence of a global attractor of optimal regularity. The proof is based on 
a new energy estimate in Sobolev spaces to bootstrap the regularity
to the optimal level, derived by means of nonlinear lower bounds on the fractional laplacian. This estimate
appears to be new in the literature, and allows a sharp use of the subcritical nature of the $L^\infty$ bounds
for this problem. As a byproduct, we obtain attractors for weak solutions as well. Moreover, we study the critical
limit of the attractors and prove their stability and upper-semicontinuity with respect to the strength of the diffusion.
\end{abstract}


\maketitle


\section{Introduction}
The dissipative surface  quasi-geostrophic equation (SQG) describes
the evolution of the potential temperature $\theta$ on the two-dimensional horizontal boundaries of general three-dimensional
quasi-geostrophic equations \cites{CMT94,P87}. Due to its similarities with the
three-dimensional Euler and Navier-Stokes equations, it has attracted the attention
of many mathematicians over the last two decades. 
Formulated on the two-dimensional torus $\TT=[0,1]^2$, the Cauchy problem reads
\begin{equation}\tag{SQG$_\gamma$}\label{eq:SQGgamma}
\begin{cases}
\de_t\theta +\u \cdot \nabla \theta+\Lambda^\gamma\theta=f,\\
\u = \RR^\perp \theta = \nabla^\perp \Lambda^{-1} \theta,\\
\theta(0)=\theta_0,\quad \int_{\TT}\theta_0(x)=0.
\end{cases}
\end{equation}
Here, $\Lambda=\sqrt{-\Delta}$ is the Zygmund operator, $\gamma\in(0,2)$ is a parameter 
measuring the strength of the diffusion, for which
the diffusivity parameter has been normalized to 1,
and $f$ is a time-independent, mean-free forcing term.
In this note, we will focus on the so-called \emph{subcritical} case, when $\gamma\in (1,2)$,
and prove the following result.

\begin{theorem}\label{thm:globalattra}
Let $\gamma\in (1,2)$ be fixed, and assume that
$f\in L^\infty\cap H^{2-\gamma}$. The dynamical system $S_\gamma(t)$ generated by \eqref{eq:SQGgamma} 
on $H^{2-\gamma}$ possesses a unique invariant global attractor $A_\gamma$, bounded in $H^{2-\gamma/2}$, and 
therefore compact in $H^{2-\gamma}$. In particular,
\begin{equation}
\lim_{t\to \infty}\dist_{H^{2-\gamma}}(S_\gamma(t)B,A_\gamma)=0,
\end{equation}
for every bounded set $B\subset H^{2-\gamma}$.
\end{theorem}
Another important feature of the attractors $A_\gamma$ is their \emph{stability} with respect to the parameter $\gamma$,
as $\gamma\to 1^+$. Namely, the following uniform estimates and upper semicontinuity result hold.


\begin{theorem}\label{thm:stability}
Let $\gamma_0\in (1,2)$ be arbitrarily fixed, and assume $f\in L^\infty\cap H^1$. 
The family of attractors $\{A_\gamma\}_{\gamma\in (1,\gamma_0]}$ is uniformly bounded
with respect to $\gamma$ in $H^{2-\gamma/2}$, namely, there exists a constant $C_0=C_0(\gamma_0, f)>0$ such that
\begin{align}
\sup_{\gamma \in (1,\gamma_0]} \|A_\gamma\|_{H^{2-\gamma/2}}\leq C_0.
\end{align} 
Moreover,   $A_\gamma$ has uniformly finite fractal dimension in $H^{2-\gamma}$, that is, there exists a constant $D_0=D_0(\gamma_0, f)>0$ such that
\begin{align}
\sup_{\gamma \in (1,\gamma_0]} \dim_{H^{2-\gamma}}A_\gamma\leq D_0.
\end{align} 
Finally, the family of attractors $\{A_\gamma\}_{\gamma\in (1,\gamma_0]}$ is 
upper semicontinuous as $\gamma\to 1^+$. Precisely,
\begin{align}\label{eq:stability}
\lim_{\gamma\to1^+}\dist_{H^1}(A_\gamma,A_1)=0,
\end{align}
where $A_1\subset H^1$ is the global attractor for the critical SQG equation (when $\gamma=1$).
\end{theorem} 
Notice that the attractors $A_\gamma$
are slightly less regular than $A_1$, being attractors in the phase space $H^{2-\gamma}$, which strictly contains $H^1$.
They are nonetheless bounded in $H^{2-\gamma/2}$ (see Theorem \ref{thm:globalattra}) and  
it is essential to have bounds on $\|A_\gamma\|_{H^{2-\gamma/2}}$ that are independent of $\gamma$. The restriction to
$\gamma\in (1,\gamma_0]$ is solely due to the use of the integral representation of the fractional Laplacian, whose normalization
constant blows up as $\gamma\to 2^-$, while the assumption $f\in L^\infty\cap H^1$ is dictated by the results valid for the critical
SQG equation. Moreover, the uniform fractal dimension estimate improves that of \cite{WT13}, where an estimate which blows up as $\gamma\to 1^+$ was proved.

The analysis can actually be extended to weak solutions to show that the basin of attraction of $A_\gamma$
is the whole space $L^2$, modulo working with multivalued dynamical systems, due to the possible non-uniqueness 
of weak solutions.
Notice also that the assumptions on $f$ can be relaxed.

\begin{theorem}\label{thm:weakglobalattra}
Let $\gamma\in (1,2)$ be fixed, and assume that
$f\in L^\infty\cap H^{\gamma/2}$. The multivalued dynamical system $S_\gamma(t)$ 
generated by \eqref{eq:SQGgamma} 
on $L^2$ possesses a unique invariant global attractor $A_\gamma$, bounded in $H^{\gamma/2}$, and therefore compact in $L^2$. In particular,
\begin{equation}\label{eq:weakat}
\lim_{t\to \infty}\dist_{L^2}(S_\gamma(t)B,A_\gamma)=0,
\end{equation}
for every bounded set $B\subset L^2$. Furthermore, if $f\in H^{2-\gamma}$, $A_\gamma$ coincides with that
of Theorem \ref{thm:globalattra}.
\end{theorem}

In the statements above, $\dist_{X}$ stands for the Hausdorff semidistance in 
$X$ between sets, given by
$$
\dist_X(B,C)=\sup_{b\in B}\inf_{c\in C}\|b-c\|_X, \quad B,C\subset X.
$$
The asymptotic behavior of solutions to \eqref{eq:SQGgamma} in terms of attractors has been investigated by
several authors in recent times.  In the subcritical case $\gamma\in (1,2)$, the existence of a weak 
global attractor in $L^2$
was proved in \cite{B02}, that is, the existence of a weakly compact, weakly attracting set for which \eqref{eq:weakat}
is replaced by the distance induced by the weak $L^2$-metric on bounded sets. A strong (and smooth) attractor was later constructed in \cite{J05}, where the semigroup $S_\gamma(t)$ was considered on $H^s$, with $s>2-\gamma$,
a space above the scale-invariant regularity level.
(see \cite{J05}*{Theorem 5.1}). The main obstructions with working in the larger space 
$H^{2-\gamma}$ can be summarized as follows.

\medskip

\noindent $\diamond$ Scaling invariance: if $\theta(x,t)$ is a solution to \eqref{eq:SQGgamma} with datum $\theta_0(x)$,  
then $\theta_\lambda(x,t) = \lambda^{\gamma-1} \theta( \lambda x,\lambda^\gamma t)$ is a 
solution of \eqref{eq:SQGgamma} with initial datum $\theta_{0,\lambda}(x) = \lambda^{\gamma-1} \theta(\lambda x)$. 
Therefore $H^{2-\gamma}$ is scale-invariant, and thus the time of local existence of a solution arising from an initial datum $\theta_0 \in H^{2-\gamma}$ is not known to depend solely on $\|\theta_0\|_{H^{2-\gamma}}$, and a uniform 
regularization with respect to initial data cannot be obtained only by exploiting short-time parabolic regularization.

\medskip

\noindent $\diamond$ Maximum principles: while smooth solutions to \eqref{eq:SQGgamma} automatically satisfy an a priori
$L^\infty$ bound, our case necessitates a uniform (w.r.t to initial data) regularization from $L^2$ to $L^\infty$, reminiscent
of De Giorgi type iterations \cites{CV10a,CD14, CD15,Sil10a}, to obtain an $L^\infty$ absorbing set (cf.~Theorem~\ref{thm:absLinf}, proven in \cite{CD15}). 

\medskip

\noindent$\diamond$ Sobolev estimates: the proofs of Theorem \ref{thm:globalattra} and \ref{thm:weakglobalattra} rely on the existence of regular absorbing sets (i.e. bounded in higher order Sobolev spaces) for the dynamics of \eqref{eq:SQGgamma}. 
However, the scaling invariance of $H^{2-\gamma}$ does not allow the use of the commutator estimates used in \cite{J05}
(see Section \ref{sec:sob}), and a new approach based on pointwise lower bounds on the fractional laplacian 
\cites{CV12,CTV13} is required (cf. Theorem \ref{prop:sob}). Specifically, the subcritical nature of the 
$L^\infty$ control of Theorem~\ref{thm:absLinf} is used in a sharp way.

\medskip

\noindent$\diamond$ Uniform estimates with respect to $\gamma$:
by exploiting only the $L^\infty$ maximum principle, the radii of the absorbing sets inevitably blow up
as $\gamma\to 1^+$.
The reason for this is fairly easy to explain: due to the scale-invariance of the $L^\infty$ norm 
in the critical case ($\gamma=1$), the existence of an $H^1$ absorbing set for $S_1(t)$
requires the existence of a $C^\beta$ absorbing set, for some $\beta\in(0,1)$ small. Here, no $C^\beta$ estimate is needed
in principle, as the $L^\infty$ norm provides a strong enough control. By adapting  
the techniques of \cite{CCZV15}, we can also prove the existence of an absorbing set consisting of 
H\"older continuous functions (cf. Proposition \ref{thm:Cbeta}), thus leading to a better choice of absorbing sets in Sobolev spaces (cf. Theorem \ref{thm:H1absbeta}), at the cost of significantly more involved 
estimates.

\medskip

It is worth mentioning that similar results hold for the critical ($\gamma=1$) SQG equation
\cites{CD14, CCZV15, CZK15, CTV13}. In a sense, the approach here generalizes  all these results
to the case $\gamma\in [1,2)$, in view of the uniformity sought in $\gamma$ of the results above.

\subsection*{Organization of the paper} In Section \ref{sec:dynamo} we introduce the proper functional setting and state
a result on
the existence of an $L^\infty$ absorbing set, proved in \cite{CD15} . We then derive a new Sobolev estimate in Section 
\ref{sec:sob}, based on pointwise estimates on the evolution of finite differences, and prove the existence of a bounded
absorbing set in $H^{2-\gamma}$. The proofs of Theorems \ref{thm:globalattra} and \ref{thm:weakglobalattra} are carried
out in Section \ref{sec:glob}. Section \ref{sec:unifabs} is dedicated to the proof of the first part of Theorem \ref{thm:stability},
dealing with the uniformity of the estimates with respect to $\gamma$, while we leave the upper-semicontinuity of 
the attractors to Section \ref{sec:stable}.

\section{The subcritical SQG equation as a dynamical system}\label{sec:dynamo}
Let $\gamma\in (1,2)$ and assume that {$f\in L^\infty\cap H^{2-\gamma}$}. It follows from several works
\cites{CW99,Dong10,Ju07,R95} that 
for all initial data $\theta_0\in H^{2-\gamma}$ the initial value problem \eqref{eq:SQGgamma}
admits a unique  global solution
$$
\theta^\gamma\in C([0,\infty);H^{2-\gamma})\cap L^2_{loc}(0,\infty;H^{2-\gamma/2}).
$$
In other words, the solution operators
$$
S_\gamma(t):H^{2-\gamma}\to H^{2-\gamma}, \qquad t\geq 0,
$$
acting as
$$
\theta_0\mapsto S_\gamma(t)\theta_0=\theta^\gamma(t), \qquad  \forall t\geq0,
$$
are well-defined and, being the forcing term autonomous, they form a semigroup of operators. By standard arguments, 
it is not hard to see that $\theta^\gamma$ satisfies the energy inequality
\begin{equation}\label{eq:energyin}
\| \theta^\gamma(t)\|^2_{L^2}+\int_0^t \|\Lambda^{\gamma/2}\theta^\gamma(s)\|_{L^2}^2\dd s\leq \|\theta_0\|^2_{L^2}
+\frac{1}{\kappa}\|f\|^2_{L^2}t, \qquad \forall t\geq 0.
\end{equation}
and the decay estimate
\begin{equation}\label{eq:expdecayL2}
\|\theta^\gamma(t)\|_{L^2}\leq \|\theta_0\|_{L^2}\e^{-\kappa t}+\frac{1}{\kappa}\|f\|_{L^2}, \qquad \forall t\geq 0,
\end{equation}
where $\kappa\geq 1$ is a universal constant independent of $\gamma$. If furthermore $\theta_0\in L^\infty$, then cf.~\cites{CC04,CTV13} we have
\begin{equation}\label{eq:expdecayLinf}
\|\theta^\gamma(t)\|_{L^\infty}\leq \|\theta_0\|_{L^\infty}\e^{-\kappa t}+\frac{1}{\kappa}\|f\|_{L^\infty} , \qquad \forall t\geq 0.
\end{equation}
Since we consider mean-zero solutions to \eqref{eq:SQGgamma}, by the symbol $H^s$ we indicate the
the \emph{homogeneous} Sobolev space of order $s\in\R$, with norm $\|\cdot\|_{H^s}=\|\Lambda^s\cdot\|_{L^2}$.
As for the fractional laplacian, we will mainly use its representation as the singular integral 
\begin{align}\label{eq:fraclapc}
\Lambda^\gamma \theta(x)= c_\gamma \sum_{k \in \ZZ^2}  \int_{\TT}   \frac{\theta(x) - \theta(x+y)}{|y- 2\pi k|^{2+\gamma}} \dd y= c_\gamma \,\mathrm{P.V.}\int_{\R^2} \frac{\theta(x)-\theta(x+y)}{|y|^{2+\gamma}}\dd y, 
\end{align} 
abusing notation and  denoting by $\theta$ the periodic extension of $\theta$ to the whole space. 
The velocity vector field $\u$ in \eqref{eq:SQGgamma} is divergence-free and determined by $\theta$ through the relation
$$
\u=\RR^\perp \theta = \nabla^\perp\Lambda^{-1}\theta=(-\de_{x_2}\Lambda^{-1}\theta,\de_{x_1}\Lambda^{-1}\theta) =(-\RR_2\theta,\RR_1 \theta),
$$
where 
\begin{align*}
\RR_j \theta(x)&= \frac{1}{2\pi} \mathrm{P.V.} \int_{\TT} \frac{y_j}{|y|^3} \theta(x+y) \dd y + \sum_{k \in \ZZ^2_*} \int_{\TT} \left( \frac{y_j + 2 \pi k_j }{|y + 2\pi k |^3}  - \frac{2 \pi k_j }{|2\pi k |^3} \right) \theta(x+y) \dd y \\
&=\frac{1}{2\pi}\,\mathrm{P.V.}\int_{\R^2}\frac{y_j}{|y|^3}\theta(x+y)\dd y.
\end{align*}
In the last line the principal value is taken both as $|y| \to 0$ and $|y|\to \infty$.

\subsection*{Constants and notation}
Throughout the paper, $c$ will denote a \emph{generic} positive constant,  
whose value may change even in the same line of a certain equation. In the 
same spirit, $c_0,c_1,\ldots$ will denote fixed constants appearing in the course of proofs 
or estimates, which have to be referred to specifically. Unless explicitly mentioned, all these
constants will be \emph{independent} of $\gamma$. The dependence on $\gamma$ of any quantity 
will be emphasized only as $\gamma\to 1^+$, while we will not worry about the case $\gamma\to 2^-$,
for which some constants are not well behaved (the constant $c_\gamma$ in \eqref{eq:fraclapc} is the only instance
of this behavior).

\subsection{$L^\infty$ absorbing sets}
Recall that a set $B_0$ 
is \emph{absorbing} if for every bounded set $B\subset H^{2-\gamma}$ there exists $t_B>0$ such that
$$
S_\gamma(t)B\subset B_0, \qquad \forall t\geq t_B.
$$
The following theorem was proved in \cite{CD15}.

\begin{theorem}\label{thm:absLinf}
Let
$$
R_\infty=\frac{2}{\kappa}\|f\|_{L^\infty}.
$$
The set
$$
B^\gamma_{\infty}=\left\{\varphi\in L^\infty\cap H^{2-\gamma}: 
\|\varphi\|_{L^\infty}\leq R_\infty\right\}
$$
is an absorbing set for $S_\gamma(t)$. Moreover, 
\begin{equation}\label{eq:estimateLinf}
\sup_{t\geq 0}\sup_{\theta_0\in B^\gamma_\infty}\|S_\gamma(t)\theta_0\|_{L^\infty}\leq 2R_\infty.
\end{equation}
\end{theorem}
The idea of the proof relies on the dissipative estimate \eqref{eq:expdecayL2} and an appropriate De Giorgi type iteration
scheme, and it is carried out in details in \cite{CD15}*{Lemma 4.2} (see also \cites{CV10a, CD14}). In particular,
it is crucial that the $L^\infty$ norm of the solution at any positive time is controlled by the $L^2$ norm of the initial datum
and the forcing. 

\begin{remark}
The assumptions on $f$ can in fact be relaxed to $f\in L^2$ at this stage, at the cost of introducing a dependence on
$\gamma$ in the expression of $R_\infty$ above. In this way, the radius of the $L^\infty$ absorbing set would diverge
as $\gamma\to1^+$.
\end{remark}

\section{Sobolev estimates via nonlinear lower bounds}\label{sec:sob}
The main goal of this section is to establish a proper dissipative estimate
of $S_\gamma(t)$ in the phase space norm $\|\cdot\|_{H^{2-\gamma}}$.
This is not at all trivial. Indeed, testing equation \eqref{eq:SQGgamma} with
$\theta$ in $H^{2-\gamma}$ and using commutator estimates (see e.g. \cite{Ju04}), we 
arrive at a differential inequality of the form
\begin{equation}\label{eq:sad}
\ddt \|\theta\|_{H^{2-\gamma}}^2 +\|\theta\|^2_{H^{2-\gamma/2}}\leq 
c\|\theta\|_{H^{2-\gamma}}\|\theta\|_{H^{2-\gamma/2}}^2+ c\|f\|^2_{H^{2-\gamma}},
\end{equation}
which does not yield any proper dissipative estimate for $\|\theta\|_{H^{2-\gamma}}$. To overcome this difficulty, 
we first proceed by pointwise estimates in the spirit of \cites{CCZV15,CZV14,CTV13,CV12}, in order to be able to exploit
the available nonlinear lower bounds for fractional diffusion operators. The main result of this section can be phrased 
as follows.

\begin{theorem}\label{prop:sob}
Let $\theta_0\in L^\infty$, $f\in L^\infty\cap H^\alpha$, $\gamma \in(1,2)$ and $\alpha\in (0,1)$. Then the differential inequality
\begin{equation}\label{eq:sob}
\ddt\|\theta\|^2_{H^\alpha}+ \frac14\|\theta\|^2_{H^{\alpha+\gamma/2}}\leq c\left[\|\theta_0\|_{L^\infty}+\frac{1}{\kappa}\|f\|_{L^\infty}\right]^{\frac{4\gamma}{\gamma-1}}+  
c\|f\|_{H^{\alpha}}^2
\end{equation}
holds for every $t>0$, with $c>0$ independent of $\gamma$.
\end{theorem}
In the case $\alpha=2-\gamma$, the  improvement of \eqref{eq:sob} with respect to \eqref{eq:sad} is dramatic, since
we now have
\begin{equation}\label{eq:sob2}
\ddt\|\theta\|^2_{H^{2-\gamma}}+ \frac14\|\theta\|^2_{H^{2-\gamma/2}}\leq c\left[\|\theta_0\|_{L^\infty}+\frac{1}{\kappa}\|f\|_{L^\infty}\right]^{\frac{4\gamma}{\gamma-1}}+  
c\|f\|_{H^{2-\gamma}}^2.
\end{equation}
The above estimate makes the scale-invariant space $H^{2-\gamma}$ treatable.
Before proceeding to the proof, postponed in Section \ref{sub:sob}, we discuss in the next section an important
consequence of the above inequality.

\subsection{Absorbing sets in scale-invariant spaces}
From estimate \eqref{eq:sob2} and the standard Gronwall lemma, we infer that
\begin{equation}\label{eq:dissipest}
\|S_\gamma(t)\theta_0\|^2_{H^{2-\gamma}}\leq \|\theta_0\|^2_{H^{2-\gamma}}\e^{-\nu t}+c\left[\|\theta_0\|_{L^\infty}+\frac{1}{\kappa}\|f\|_{L^\infty}\right]^{\frac{4\gamma}{\gamma-1}}+  
c\|f\|_{H^{2-\gamma}}^2,
\end{equation}
where $\nu>0$ depends on the Poincar\'e constant and can clearly be made independent of $\gamma\in (1,2)$.
In particular, due to the existence of an $L^\infty$ absorbing set (Theorem \eqref{thm:absLinf}), the
existence of an $H^{2-\gamma}$ absorbing set follows immediately.

\begin{theorem}\label{thm:H1abs}
The set
\begin{equation*}
B_1^\gamma=\left\{\varphi\in H^{2-\gamma}: \|\varphi\|_{H^{2-\gamma}}\leq R_{1,\gamma}\right\},
\end{equation*}
with
$$
R_{1,\gamma}^2=c\left[2R_\infty\right]^{\frac{4\gamma}{\gamma-1}}+  
c\|f\|_{H^{2-\gamma}}^2,
$$
is absorbing for $S_\gamma(t)$. Moreover, 
\begin{equation}\label{eq:H1unif}
\sup_{t\geq 0}\sup_{\theta_0\in B^\gamma_1}\left[\|S_\gamma(t)\theta_0\|^2_{H^{2-\gamma}}+\int_t^{t+1}\|S_\gamma(\tau)\theta_0\|^2_{H^{2-\gamma/2}}\dd \tau\right]\leq 
2 R_{1,\gamma}^2.
\end{equation}
\end{theorem}
Estimate \eqref{eq:H1unif} is derived by choosing an initial datum $\theta_0\in  B_1^\gamma$, integrating on $(t,t+1)$ inequality \eqref{eq:sob2} and exploiting the bound \eqref{eq:dissipest}. We discuss the optimality (rather, the non-optimality) of 
the radius $R_{1,\gamma}$ in Section \ref{sec:unifabs}.

\begin{remark}
In \cite{CTV13}, an estimate of similar flavor was derived in the case $\gamma=\alpha=1$ by considering the 
evolution of $\nabla \theta$ and exploiting H\"older bounds. The approach here is somewhat different, for two
main reasons linked to the nonlocal nature of $\Lambda$: firstly, the evolution of $\Lambda^\alpha\theta$ 
is not as nice, as Leibniz differentiation does not hold anymore; secondly, the pointwise nonlinear lower bounds
hold for $\nabla \theta$, but it is not clear whether they hold for $\Lambda^\alpha\theta$ or not. We refer
to \cite{CZK15} for an estimate involving H\"older norms.
\end{remark}

\subsection{A general Sobolev estimate}\label{sub:sob}
For convenience, in the course of this section we will set
\begin{equation}\label{eq:globalLinf}
K_\infty=\|\theta_0\|_{L^\infty}+\frac{1}{\kappa}\|f\|_{L^\infty},
\end{equation}
so that in view of \eqref{eq:expdecayLinf} the solution originating from $\theta_0$ satisfies the global bound
\begin{equation}\label{eq:globalLinf2}
\|\theta(t)\|_{L^\infty}\leq K_\infty, \qquad \forall t\geq 0.
\end{equation}
Consider  the finite difference
\begin{align*}
\delta_h\theta(x,t)=\theta(x+h,t)-\theta(x,t),
\end{align*}
which is periodic in both $x$ and $h$, where  $x,h \in \TT$.  In turn, 
\begin{align}\label{eq:findiff0}
L (\delta_h\theta)=\delta_h f,
\end{align}
where $L$ denotes the differential operator
$$
L=\de_t+\u\cdot \nabla_x+(\delta_h\u)\cdot \nabla_h+ \Lambda^\gamma.
$$
From \eqref{eq:findiff0}, we use the formula (see \cite{CC04})
\begin{align*}
2\varphi(x) \Lambda^\gamma \varphi(x)=\Lambda^\gamma \big(\varphi(x)^2\big)+D_\gamma[\varphi](x), 
\end{align*}
valid for $\gamma\in (0,2)$ and $\varphi\in C^\infty(\TT)$, and with
\begin{equation}
\label{eq:D:gamma:def}
D_\gamma[\varphi](x)= c_\gamma \int_{\R^2} \frac{\big[\varphi(x)-\varphi(x+y)\big]^2}{|y|^{2+\gamma}}\dd y.
\end{equation}
We then arrive at
\begin{equation}\label{eq:findiff}
L (\delta_h\theta)^2+ D_\gamma[\delta_h\theta]=2(\delta_h f) (\delta_h\theta).
\end{equation}
For an arbitrary $\alpha\in (0,1)$,
we study the evolution of the quantity $v(x,t;h)$ defined by
$$
v(x,t;h) =\frac{\delta_h\theta(x,t)}{|h|^{1+\alpha}}.
$$
Notice that $v$ is very much related to the usual fractional Sobolev norms, in the sense that
$$
\|\theta(t)\|^2_{H^\alpha}=\int_{\R^2}\int_{\R^2}\big[v(x,t;h)\big]^2\dd h\, \dd x= 
\int_{\R^2}\int_{\R^2}\frac{\big[\theta(x+h,t)-\theta(x,t)\big]^2}{|h|^{2+2\alpha}}\dd h\, \dd x.
$$
From \eqref{eq:findiff}, we deduce that
\begin{align}
L v^2+\frac{ D_\gamma[\delta_h\theta] }{|h|^{2+2\alpha}}
&=-2(1+\alpha) \frac{h}{|h|^2}\cdot \delta_h\u \, v^2+
2\frac{(\delta_hf)(\delta_h\theta)}{|h|^{2+2\alpha}}, \label{eq:ineq1}
\end{align}
with $\delta_h\u= \RR^\perp \delta_h\theta$, namely, the perpendicular Riesz transform of the scalar $\delta_h\theta$. 
We now estimate the dissipative 
term $D_\gamma[\delta_h\theta]$ from below and the drift term containing $\delta_h\u$
from above. The proofs of the next Lemmas are very similar to those contained in  \cites{CV12, CTV13,CZV14}, but we 
report them here for the sake of completeness.

\begin{lemma}\label{lem:nonlinbdd}
There exists a positive constant
$\tilde{c}_\gamma$ such that
$$
D_\gamma[\delta_h\theta](x,t)\geq \tilde{c}_\gamma \frac{|\delta_h\theta(x,t)|^{2+\gamma}}{|h|^\gamma\|\theta(t)\|^\gamma_{L^\infty}}
$$
holds for any $x,h\in \TT$ and any $t\geq 0$.
\end{lemma}

\begin{proof}
For the sake of brevity, we omit the time dependence of every function below. 
Let $\chi$ be a smooth radially non-increasing cutoff function that vanishes on $|x|\leq 1$ and is identically 1 for $|x|\geq 2$ and
such that $|\chi'|\leq 2$. For $r\geq 4|h|$, we estimate
\begin{align}\label{eq:longcalc}
D_\gamma[\delta_h\theta](x)
&\geq c_\gamma \int_{\R^2} \frac{\big[\delta_h\theta(x)-\delta_h\theta(x+y)\big]^2}{|y|^{2+\gamma}}\chi(|y|/r)\dd y \notag \\
&\geq c_\gamma |\delta_h\theta(x)|^2\int_{\R^2} \frac{\chi(|y|/r)}{|y|^{2+\gamma}}\dd y 
-2c_\gamma |\delta_h\theta(x)| \left|  \int_{\R^2} \frac{\delta_h\theta(x+y)}{|y|^{2+\gamma}}\chi(|y|/r)\dd y        \right| \notag \\
&\geq  c_\gamma \frac{|\delta_h\theta(x)|^2}{r^\gamma} 
-2c_\gamma |\delta_h\theta(x)| \left|  \int_{\R^2} \big[\theta(x+y)-\theta(x)\big]\delta_{-h} \frac{\chi(|y|/r)}{|y|^{2+\gamma}}\dd y \right| \notag \\
&\geq  c_\gamma \frac{|\delta_h\theta(x)|^2}{r^\gamma} 
-c_1c_\gamma |\delta_h\theta(x)|  \, |h|   
\int_{|y|\geq r} \frac{|\delta_y\theta(x)|}{|y|^{3+\gamma}} \dd y \notag \\
&\geq  c_\gamma \frac{|\delta_h\theta(x)|^2}{r^\gamma} -2c_1c_\gamma |\delta_h\theta(x)|  \, |h|  \, \|\theta\|_{L^\infty}
\int_{r}^\infty  \frac{1}{\rho^{2+\gamma}}\dd \rho,
\end{align}
for some constant $c_1\geq 1$.
Hence, for $r\geq 4|h|$ there holds
\begin{equation}\label{eq:dissip1}
D_\gamma[\delta_h\theta](x)\geq \frac{c_\gamma}{r^\gamma}|\delta_h\theta(x)|^2-c c_\gamma|\delta_h\theta(x)|\|\theta\|_{L^\infty}\frac{|h|}{r^{1+\gamma}},
\end{equation}
where $c\geq 1$ is an absolute constant. We choose $r>0$ such that
$$
\frac{c_\gamma}{r^\gamma}|\delta_h\theta(x)|^2=8c c_\gamma|\delta_h\theta(x)|\|\theta\|_{L^\infty}\frac{|h|}{r^{1+\gamma}},
$$
namely,
$$
r=\frac{8c\|\theta\|_{L^\infty}|h|}{|\delta_h\theta(x)|}.
$$
Notice that since $|\delta_h\theta(x)|\leq 2\|\theta\|_{L^\infty}$, we immediately obtain that $r\geq 4|h|$.
The result follows by plugging $r$ back into \eqref{eq:dissip1}.
\end{proof}

Concerning the nonlinear term, we have the following.

\begin{lemma}\label{lem:rieszbdd}
Let $r\geq 4|h|$ be arbitrarily fixed. Then
$$
|\delta_h\u(x,t)|\leq 
c\left[ r^{\gamma/2} \big(D_\gamma[\delta_h\theta](x,t)\big)^{1/2}+\frac{|h|\|\theta(t)\|_{L^\infty} }{r}\right],
$$
holds pointwise in $x,h\in \TT$ and $t\geq0$.
\end{lemma}
\begin{proof}
Let us fix $r\geq 4|h|$, and let $\chi$ be a smooth radially non-increasing cutoff function 
that vanishes on $|x|\leq 1$ and is identically 1 for $|x|\geq 2$ and such that $|\chi'|\leq 2$.
We split the vector $\delta_h\u$ in an inner and an outer part
\begin{align*}
\delta_h\u(x)
=\frac{1}{2\pi} \mathrm{P.V.}\int_{\R^2} \frac{y^\perp}{|y|^3}\big[\delta_h\theta(x+y)-\delta_h\theta(x)\big]\dd y
=\delta_h\u_{in}(x)+\delta_h\u_{out}(x),
\end{align*}
by using that the kernel of $\RR^\perp$ has zero average on the unit sphere, where
\begin{align*}
\delta_h\u_{in}(x)=\frac{1}{2\pi}\mathrm{P.V.}\int_{\R^2} \frac{y^\perp}{|y|^3}\big[1-\chi(|y|/r)\big] \big[\delta_h\theta(x+y)-\delta_h\theta(x)\big]\dd y,
\end{align*}
and
\begin{align*}
\delta_h\u_{out}(x)&=\frac{1}{2\pi}\mathrm{P.V.}\int_{\R^2} \frac{y^\perp}{|y|^3} \chi(|y|/r)\big[\delta_h\theta(x+y)\big]\dd y\\
&=\frac{1}{2\pi}\mathrm{P.V.}\int_{\R^2} \delta_{-h}\left[\frac{y^\perp}{|y|^3}\chi(|y|/r) \right]\big[\theta(x+y)\big]\dd y.
\end{align*}
For the inner piece, in light of the Cauchy-Schwartz inequality, we obtain
\begin{align}
|\delta_h\u_{in}(x)|&\leq\frac{1}{2\pi}\int_{|y|\leq r} \frac{1}{|y|^2}|\delta_h\theta(x+y)-\delta_h\theta(x)|\dd y
\notag \\
&\leq \frac{1}{2\pi} \left[ \int_{|y|\leq r} \frac{1}{|y|^{2-\gamma}}  \right]^{1/2}\left[  \int_{\R^2}\frac{(\delta_h\theta(x+y)-\delta_h\theta(x))^2}{|y|^{2+\gamma}} \dd y\right]^{1/2}
\notag \\
&\leq cr^{\gamma/2} \big(D_\gamma[\delta_h\theta](x)\big)^{1/2}.
\label{eq:pfriesz1}
\end{align}
Regarding the outer part, the mean value theorem entails
\begin{align}\label{eq:pfriesz2}
|\delta_h\u_{out}(x)|&\leq c|h|\int_{|y|\geq r/2} \frac{|\theta(x+y)|}{|y|^3}\dd y
\leq c\frac{|h|\|\theta\|_{L^\infty}}{r}.
\end{align}
The conclusion follows by combining \eqref{eq:pfriesz1} and \eqref{eq:pfriesz2}. 
\end{proof}
We are now ready to complete the proof of the estimate \eqref{eq:sob}.

\begin{proof}[Proof of Theorem \ref{prop:sob}]
Without loss of generality, we may assume $K_\infty\geq1$.
Combining  \eqref{eq:globalLinf} and \eqref{eq:ineq1} with the results of the above two lemmas, we obtain the inequality
\begin{align}
L v^2+\frac12\frac{ D_\gamma[\delta_h\theta] }{|h|^{2+2\alpha}}+
\tilde{c}_\gamma \frac{|\delta_h\theta|^{2+\gamma}}{|h|^{2+2\alpha+\gamma} K_\infty^\gamma}
&\leq c\left[ r^{\gamma/2} \big(D_\gamma[\delta_h\theta]\big)^{1/2}+\frac{|h|K_\infty }{r}\right]\frac{ v^2}{|h|}  +
2\frac{(\delta_hf)(\delta_h\theta)}{|h|^{2+2\alpha}}. \label{eq:ineq2}
\end{align}
By the Cauchy-Schwartz inequality,
$$
c\left[ r^{\gamma/2} \big(D_\gamma[\delta_h\theta]\big)^{1/2}
+\frac{|h|K_\infty}{r}\right]\frac{ v^2}{|h|} \leq \frac14\frac{ D_\gamma[\delta_h\theta]}{|h|^{2+2\alpha}} +
c \left[|h|^{2\alpha}v^4r^\gamma+ \frac{K_\infty }{r}v^2\right].
$$
We now choose $r>0$ as
$$
r=4\left[\frac{4K_\infty^2}{|h|^{2\alpha}v^2}\right]^{\frac{1}{1+\gamma}},
$$
so that, in particular by \eqref{eq:globalLinf2} we obtain
$$
r=4\left[\frac{4K_\infty^2}{|\delta_h\theta|^2}\right]^{\frac{1}{1+\gamma}}|h|^{\frac{2}{1+\gamma}}\geq 4|h|^{\frac{2}{1+\gamma}}\geq 4|h|,
$$
since $|h|\leq 1$ and $\gamma>1$. In this way, since we assumed $K_\infty\geq 1$,
$$
|h|^{2\alpha}v^4r^\gamma+ \frac{K_\infty}{r}v^2\leq 2|h|^{2\alpha}v^4r^\gamma \leq  c K_\infty^{\frac{2\gamma}{1+\gamma}} |h|^{\frac{2\alpha}{1+\gamma}}
v^{2+\frac{2}{1+\gamma}},
$$
and \eqref{eq:ineq2} becomes
\begin{align}
L v^2+\frac14\frac{ D_\gamma[\delta_h\theta] }{|h|^{2+2\alpha}}+
\tilde{c}_\gamma \frac{|\delta_h\theta|^{2+\gamma}}{|h|^{2+2\alpha+\gamma} K_\infty^\gamma}
&\leq c K_\infty^{\frac{2\gamma}{1+\gamma}} |h|^{\frac{2\alpha}{1+\gamma}}
v^{2+\frac{2}{1+\gamma}}  +
\frac{(\delta_hf)(\delta_h\theta)}{|h|^{2+2\alpha}}. \label{eq:ineq3}
\end{align}
Using Young inequality with 
$$
p=\frac{1+\gamma}{2}, \qquad q=\frac{1+\gamma}{\gamma-1},
$$
we infer that
$$
c K_\infty^{\frac{2\gamma}{1+\gamma}} |h|^{\frac{2\alpha}{1+\gamma}}
v^{2+\frac{2}{1+\gamma}}\leq 
\tilde{c}_\gamma \frac{|\delta_h\theta|^{2+\gamma}}{|h|^{2+2\alpha+\gamma} K_\infty^\gamma} + c\frac{K_\infty^{\frac{4\gamma}{\gamma-1}}}{|h|^{2\alpha}} 
$$
Therefore, from \eqref{eq:ineq3} we deduce that
$$
L v^2+\frac14\frac{ D_\gamma[\delta_h\theta] }{|h|^{2+2\alpha}}
\leq c\frac{K_\infty^{\frac{4\gamma}{\gamma-1}}}{|h|^{2\alpha}}  +
2\frac{(\delta_hf)(\delta_h\theta)}{|h|^{2+2\alpha}}.
$$
We integrate the above inequality first in $h\in \TT$ (which is allowed, since $\alpha<1$) and then $x\in \TT$. Using that 
$$
\frac12\int_{\R^2}\int_{\R^2} \frac{ D_\gamma[\delta_h\theta] }{|h|^{2+2\alpha}}\dd h\,\dd x
=\int_{\R^2}\int_{\R^2} \frac{|\delta_h\Lambda^{\gamma/2}\theta|^2}{|h|^{2+2\alpha}}\dd h\, \dd x
=\|\theta\|^2_{H^{\alpha+\gamma/2}}
$$
and the estimate, valid for $\alpha\in (0,1)$,
\begin{align*}
\int_{\R^2}\int_{\R^2}\frac{(\delta_hf)(\delta_h\theta)}{|h|^{2+2\alpha}}\dd h\, \dd x 
&\leq \left[\int_{\R^2}\int_{\R^2}\frac{|\delta_hf|^2}{|h|^{2+2\alpha}}\dd h\, \dd x \right]^{1/2}
\left[\int_{\R^2}\int_{\R^2}\frac{|\delta_h\theta|^2}{|h|^{2+2\alpha}}\dd h\, \dd x \right]^{1/2}\\
&\leq \|f\|_{H^{\alpha}}\|\theta\|_{H^{\alpha}}
\leq \frac14 \|\theta\|_{H^{\alpha+\gamma/2}}^2+c\|f\|_{H^{\alpha}}^2,
\end{align*}
we arrive at
$$
\ddt\|\theta\|^2_{H^\alpha}+ \frac14\|\theta\|^2_{H^{\alpha+\gamma/2}}\leq cK_\infty^{\frac{4\gamma}{\gamma-1}}+  
c\|f\|_{H^{\alpha}}^2.
$$
This is precisely \eqref{eq:sob}, and the proof is concluded.
\end{proof}

\section{The global attractor}\label{sec:glob}
A sufficient condition for the existence of the global attractor (the unique compact set of the phase space
that is invariant and attracting) for a dynamical system is the existence of a 
compact absorbing set \cites{Hale,SellYou,T3}. Moreover, being the attractor the minimal set in the class of closed attracting sets,
it is contained in any (closed) absorbing set. In particular, the attractor inherits the regularity property of
the absorbing set, namely, the existence of regular (i.e. bounded in higher order Sobolev spaces) absorbing sets 
translate into the existence of a regular attractor. We prove Theorem \ref{thm:globalattra} in the next Section \ref{sub:regattr},
and Theorem \ref{thm:weakglobalattra} in the subsequent Section \ref{sub:weakattr}, by using again the estimate
\eqref{eq:sob} several times.

In the course of this section, we will often make use of the fractional product inequality \cite{KP88}
\begin{equation}\label{eq:prod}
\|\Lambda^s(\varphi\psi)\|_{L^p}\leq 
c\left[\|\varphi\|_{L^{p_1}}\|\Lambda^s\psi\|_{L^{p_2}}+\|\Lambda^s\varphi\|_{L^{p_3}}\|\psi\|_{L^{p_4}}\right],
\end{equation}
valid for $s>0$, $1/p=1/p_1+1/p_2=1/p_3+1/p_4$ and $p_2,p_3\in (1,\infty)$,
the commutator estimate \cite{KPV91}
\begin{equation}\label{eq:comm}
\|\Lambda^{s} (\varphi \psi)-\varphi\Lambda^{s}\psi\|_{L^p}
\leq c\left[\|\nabla \varphi\|_{L^{p_1}}\|\Lambda^{s-1}\psi\|_{L^{p_2}}+\|\Lambda^{s}\varphi\|_{L^{p_3}}\|\psi\|_{L^{p_4}}\right],
\end{equation}
with the same constraints as above,
and the Sobolev embedding
\begin{equation}\label{eq:imbe}
\|\varphi\|_{L^p}\leq c\| \Lambda^{1-2/p}\varphi\|_{L^2},
\end{equation}
with $p\in [2,\infty)$.

\subsection{Regular absorbing sets}\label{sub:regattr}
The existence and regularity of the attractor in Theorem \ref{thm:globalattra} follow from the existence of an
absorbing set bounded in $H^{2-\gamma/2}$. 

\begin{theorem}\label{thm:H32abs}
The set
\begin{equation*}
B^\gamma_2=\left\{\varphi\in H^{2-\gamma/2}: \|\varphi\|_{H^{2-\gamma/2}}
\leq R_{2,\gamma}\right\}
\end{equation*}
with
$$
R_{2,\gamma}^2=c\left[2R^2_{1,\gamma}+ \|f\|^2_{H^{2-\gamma}}\right]\e^{cR^2_{1,\gamma}},
$$
is absorbing for $S_\gamma(t)$. Moreover, 
\begin{equation}\label{eq:H32unif}
\sup_{t\geq 0}\sup_{\theta_0\in B^\gamma_2}\left[\|S_\gamma(t)\theta_0\|^2_{H^{2-\gamma/2}}+\int_t^{t+1}\|S_\gamma(\tau)\theta_0\|^2_{H^{2}}\dd \tau\right]\leq 
\Q(R_{2,\gamma}),
\end{equation}
where $\Q(\cdot)$ is a positive increasing function with $\Q(0)=0$.
\end{theorem}

\begin{proof}
Clearly, it is enough to show that $B^\gamma_2$ absorbs $B^\gamma_1$, the $H^{2-\gamma}$
absorbing set obtained in Theorem~\ref{thm:H1abs}. If $\theta_0\in B^\gamma_1$, then
\eqref{eq:H1unif} implies that 
\begin{equation}\label{eq:H32int2}
\sup_{t\geq 0}\int_t^{t+1}\|S_\gamma(\tau)\theta_0\|^2_{H^{2-\gamma/2}}\dd \tau\leq 2 R^2_{1,\gamma}.
\end{equation}
By testing \eqref{eq:SQGgamma} with $\theta$ in $H^{2-\gamma/2}$ and using standard arguments, we deduce that
\begin{align*}
\ddt \|\theta\|^2_{H^{2-\gamma/2}}+\|\theta\|^2_{H^2}\leq \|f\|^2_{H^{2-\gamma}}+
2\left|\int_{\TT}\left[\Lambda^{2-\gamma/2}(\u\cdot\nabla\theta)-\u\cdot \nabla\Lambda^{2-\gamma/2}\theta\right]\Lambda^{2-\gamma/2}\theta \dd x\right|.
\end{align*}
By means of the commutator estimate \eqref{eq:comm},
$$
\|\Lambda^{2-\gamma/2} (\varphi \psi)-\varphi\Lambda^{2-\gamma/2}\psi\|_{L^2}
\leq c\left[\|\nabla \varphi\|_{L^{4/\gamma}}\|\Lambda^{1-\gamma/2}\psi\|_{L^{\frac{4}{2-\gamma}}}+\|\Lambda^{2-\gamma/2}\varphi\|_{L^{\frac{4}{2-\gamma}}}\|\psi\|_{L^{4/\gamma}}\right], 
$$
and the two-dimensional Sobolev inequality
$$ 
\|\varphi\|_{L^p}\leq c\|\Lambda^{1-2/p}\varphi\|_{L^2},\qquad p\in[2,\infty),
$$
we therefore have
\begin{align*}
\ddt \|\theta\|^2_{H^{2-\gamma/2}}+\|\theta\|^2_{H^2}
&\leq \|f\|^2_{H^{2-\gamma}}+
c\|\theta\|_{H^{2-\gamma/2}}\|\Lambda \theta\|_{L^{4/\gamma}}\|\Lambda^{2-\gamma/2}\theta\|_{L^\frac{4}{2-\gamma}}\\
&\leq \|f\|^2_{H^{2-\gamma}}+c\|\theta\|_{H^{2-\gamma/2}}^2\|\theta\|_{H^2}\\
&\leq \|f\|^2_{H^{2-\gamma}}+c\|\theta\|_{H^{2-\gamma/2}}^4+\frac12\|\theta\|^2_{H^2}.
\end{align*}
Hence,
\begin{align}\label{eq:high}
\ddt \|\theta\|^2_{H^{2-\gamma/2}}+\frac12\|\theta\|^2_{H^2}\leq \|f\|^2_{H^{2-\gamma}}+c\|\theta\|_{H^{2-\gamma/2}}^4.
\end{align}
Thanks to the local integrability \eqref{eq:H32int2} and the above differential
inequality, the uniform Gronwall lemma implies
\begin{equation}\label{eq:boh}
 \|S_\gamma(t)\theta_0\|^2_{H^{2-\gamma/2}}\leq c\left[2R^2_{1,\gamma}+ \|f\|^2_{H^{2-\gamma}}\right]\e^{cR^2_{1,\gamma}}, \qquad \forall t\geq 1.
\end{equation}
Thus, setting
$$
R_{2,\gamma}^2:=c\left[2R^2_{1,\gamma}+ \|f\|^2_{H^{2-\gamma}}\right]\e^{cR^2_{1,\gamma}},
$$
we obtain that 
$$
S_\gamma(t)B_{1}^\gamma\subset B_{2}^\gamma, \qquad \forall t\geq 1,
$$
as we wanted. Concerning estimate \eqref{eq:H32unif}, it is clear that it holds for $t\geq1$ from \eqref{eq:boh} and 
by integrating \eqref{eq:high} on $(t,t+1)$. For $t<1$, it suffices to use \eqref{eq:H32int2} and the standard Gronwall
lemma on the time interval $(0,t)$, applied to \eqref{eq:high}.
\end{proof}

The existence of a compact absorbing set is well-known to be sufficient for the existence of the global attractor.
However, due to the possible lack of continuity of the map $S_\gamma(t):H^{2-\gamma}\to H^{2-\gamma}$ for fixed
$t>0$, the invariance of $A_\gamma$ requires some care. In fact, to conclude the proof of Theorem \ref{thm:globalattra}, 
it is enough to prove continuity on the regular absorbing set $B_{2,\gamma}$. Our next goal is then to establish the following.

\begin{proposition}
Let $\gamma\in (1,2)$. For each fixed $t\geq 0$, $S_\gamma(t)$ is Lipschitz-continuous on $B^\gamma_2$ in the
topology of $H^{2-\gamma}$ and
\begin{equation}\label{eq:conti}
\sup_{\theta_{0,i}\in B^\gamma_2}\|S_\gamma(t) \theta_{0,1}-S_{\gamma}(t)\theta_{0,2}\|_{H^{2-\gamma}}\leq \e^{\Q(R_{2,\gamma}) t} \| \theta_{0,1}-\theta_{0,2}\|_{H^{2-\gamma}}, \qquad \forall t\geq 0,
\end{equation}
where $\Q(\cdot)$ is a positive increasing function with $\Q(0)=0$.
\end{proposition}

\begin{proof}
Denote by $\theta_i(t)=S_\gamma(t)\theta_{0,i}$, $i=1,2$, two solutions emanating from initial data $\theta_{0,i}\in B_2^\gamma$.
Their difference $\bar\theta=\theta_1-\theta_2$ solves the equation
$$
\de_t\bar\theta +\u_1\cdot \nabla \bar\theta+ \bar\u\cdot\nabla \theta_2+\Lambda^\gamma\bar\theta=0.
$$
Testing the above equation with $\bar\theta$ in $H^{2-\gamma}$ yields
\begin{equation}\label{eq:nonso}
\frac12\ddt\|\bar\theta\|^2_{H^{2-\gamma}}+\|\bar\theta\|^2_{H^{2-\gamma/2}}
\leq \left|\int_{\TT}\Lambda^{2-\gamma}(\u_1\cdot \nabla \bar\theta)\Lambda^{2-\gamma}\bar\theta\,\dd x \right|
+\left|\int_{\TT}\Lambda^{2-\gamma}(\bar\u\cdot \nabla \theta_2)\Lambda^{2-\gamma}\bar\theta\,\dd x\right|.
\end{equation}
We estimate the two terms in right-hand-side above separately. Concerning the first one,  we use \eqref{eq:comm} and 
\eqref{eq:imbe} to get
\begin{align*}
&\left|\int_{\TT}\Lambda^{2-\gamma}(\u_1\cdot \nabla \bar\theta)\Lambda^{2-\gamma}\bar\theta\,\dd x \right|
=\left|\int_{\TT}\left[\Lambda^{2-\gamma}(\u_1\cdot \nabla \bar\theta)-\u_1\cdot\Lambda^{2-\gamma}\nabla\bar\theta\right]\Lambda^{2-\gamma}\bar\theta\,\dd x \right|\\
&\qquad\qquad\leq \|\Lambda^{2-\gamma}(\u_1\cdot \nabla \bar\theta)-\u_1\cdot\Lambda^{2-\gamma}\nabla\bar\theta\|_{L^{\frac{4}{2+\gamma}}}\|\Lambda^{2-\gamma}\bar\theta\|_{L^\frac{4}{2-\gamma}}\\
&\qquad\qquad\leq c\left[\|\nabla\u_1\|_{L^{4/\gamma}} \|\Lambda^{2-\gamma}\bar\theta\|_{L^2}+
\|\Lambda^{2-\gamma}\u_1\|_{L^{\frac{4}{2-\gamma}}} \|\nabla\bar\theta\|_{L^{2/\gamma}}\right]\|\bar\theta\|_{H^{2-\gamma/2}}\\
&\qquad\qquad\leq c \|\theta_1\|_{H^{2-\gamma/2}}\|\bar\theta\|_{H^{2-\gamma}}\|\bar\theta\|_{H^{2-\gamma/2}},
\end{align*}
while for the second we exploit \eqref{eq:prod} and \eqref{eq:imbe} to obtain
\begin{align*}
&\left|\int_{\TT}\Lambda^{2-\gamma}(\bar\u\cdot \nabla \theta_2)\Lambda^{2-\gamma}\bar\theta\,\dd x\right|
=\left|\int_{\TT}\Lambda^{2-3\gamma/2}(\bar\u\cdot \nabla \theta_2)\Lambda^{2-\gamma/2}\bar\theta\,\dd x\right|\\
&\qquad\qquad\leq c\|\Lambda^{2-3\gamma/2}(\bar\u\cdot \nabla \theta_2)\|_{L^2} \|\Lambda^{2-\gamma/2}\bar\theta\|_{L^2}\\
&\qquad\qquad\leq c\left[\|\bar\u\|_{L^{\frac{2}{\gamma-1}}}\|\Lambda^{3-3\gamma/2} \theta_2\|_{L^\frac{2}{2-\gamma}}
+\|\Lambda^{2-3\gamma/2}\bar\u\|_{L^\frac{4}{2-\gamma}} \|\nabla\theta_2\|_{L^{4/\gamma}}\right] \|\bar\theta\|_{H^{2-\gamma/2}}\\
&\qquad\qquad\leq c \|\theta_2\|_{H^{2-\gamma/2}}\|\bar\theta\|_{H^{2-\gamma}}\|\bar\theta\|_{H^{2-\gamma/2}}.
\end{align*}
In view of the above estimates and using Young inequality, \eqref{eq:nonso} becomes
$$
\ddt\|\bar\theta\|^2_{H^{2-\gamma}}+\|\bar\theta\|^2_{H^{2-\gamma/2}}
\leq c\left[\|\theta_1\|_{H^{2-\gamma/2}}^2+\|\theta_2\|_{H^{2-\gamma/2}}^2\right]\|\bar\theta\|_{H^{2-\gamma}}^2.
$$
In light of \eqref{eq:H32unif}, the continuous dependence estimate \eqref{eq:conti} follows from a further application
of the Gronwall inequality. 
\end{proof}

\subsection{Global attractors for weak solutions}\label{sub:weakattr}
A viscosity solution to \eqref{eq:SQGgamma} is a mean free function 
$\theta^\gamma\in C([0,\infty);L^2)$  that satisfies \eqref{eq:SQGgamma} in 
the sense of distributions, and such that there exist sequences
$\eps_n\to 0$ and $\theta_n^\gamma$ satisfying
$$
\begin{cases}
\de_t\theta_n^\gamma +\u_n^\gamma \cdot \nabla \theta_n^\gamma+\Lambda^\gamma\theta_n^\gamma-\eps_n\Delta\theta_n^\gamma=f,\\
\u_n^\gamma = \RR^\perp \theta_n^\gamma = \nabla^\perp \Lambda^{-1} \theta_n^\gamma,
\end{cases}
$$
such that $\theta^\gamma_n\to \theta^\gamma$ in $C_w([0,T];L^2)$, for every $T>0$ and $\theta^\gamma_{n}(0)\to \theta(0)$ strongly in $L^2$. From \cite{CC04}, it follows 
that for any $\theta_0\in L^2$, a (possibly non-unique) viscosity solution to \eqref{eq:SQGgamma} 
exists. The fact that viscosity solutions are strongly continuous is a consequence of the fact that they satisfy
the energy equality (see \cites{CD14, CTV14} for a proof in the critical case). Following the approach in \cites{CZ13, MV98}, for $t\geq 0$ and each $\theta_0\in L^2$ 
we define the set-valued maps $S_\gamma(t):L^2\to 2^{L^2}$, still denoted as the single-valued ones,
$$
S_\gamma(t)\theta_0=\big\{\theta^\gamma(t): \theta^\gamma(\cdot) \text{ is a viscosity solution to \eqref{eq:SQGgamma} with } \theta^\gamma(0)=\theta_0\big\}.
$$
Similarly to the critical case investigated in \cite{CZK15}, it is possible to show that
translations and concatenations of viscosity solutions are still viscosity solutions, so that 
$S_\gamma(t)$ satisfies the semigroup property
$$
S_\gamma(t+\tau)=S_\gamma(t)S_\gamma(\tau), \qquad \forall t,\tau\geq 0.
$$
Moreover, the graph of $S_\gamma(t)$ is closed, namely for any $t\geq 0$ 
the following implication holds true: 
$$
\theta_{0,n}\to \theta_0,\quad S_\gamma(t)\theta_{0,n} \ni \theta_n\to \theta \qquad \Rightarrow \qquad \theta\in S_\gamma(t)\theta_0.
$$
Above, limits are understood in the strong topology of $L^2$.
Therefore, to prove the existence of the global attractor is again sufficient to exhibit a compact absorbing set. To begin
with, \eqref{eq:expdecayL2} implies the existence of an $L^2$ bounded absorbing set
$$
B_0^\gamma=\left\{\varphi\in L^2: 
\|\varphi\|_{L^2}\leq \frac{2}{\kappa}\|f\|_{L^2}\right\}.
$$ 
In addition, it is not hard to see from \eqref{eq:energyin}, which holds for viscosity solutions (cf. \cite{CD14}), 
we also gain time integrability whenever $\theta_0\in B_0^\gamma$, namely 
\begin{equation}\label{eq:timeintg2}
\sup_{t\geq 0}\int_t^{t+1}\|S_\gamma(\tau)\theta_0\|^2_{H^{\gamma/2}}\dd \tau\leq 
\frac{4}{\kappa}\|f\|_{L^2}^2,
\end{equation}
where $\|S_\gamma(\tau)\theta_0\|^2_{H^{\gamma/2}}$ has to be understood as the supremum over all the
elements in the set $S_\gamma(\tau)\theta_0$.
Once a uniform $L^2$ estimate is available, we can proceed as in Section \ref{sec:dynamo} and deduce that
the set $B_\infty^\gamma$ in \eqref{eq:estimateLinf} defines an $L^\infty$ absorbing set for the multivalued case as well. 
It is crucial here that the norm of $\theta_0$ in $L^2$ controls the $L^\infty$ norm of the solutions for all positive times.
As the next step, we can assume $\theta_0 \in B_\infty^\gamma$ and apply \eqref{eq:sob} with $\alpha=\gamma/2$, that is
\begin{equation}\label{eq:sob3}
\ddt\|\theta\|^2_{H^{\gamma/2}}+ \frac14\|\theta\|^2_{H^{\gamma}}\leq c\left[\|\theta_0\|_{L^\infty}+\frac{1}{\kappa}\|f\|_{L^\infty}\right]^{\frac{4\gamma}{\gamma-1}}+  
c\|f\|_{H^{\gamma/2}}^2.
\end{equation}
By neglecting the positive term $\|\theta\|^2_{H^{\gamma}}$, using \eqref{eq:timeintg2} and the uniform Gronwall lemma
we have that
$$
\|S_\gamma(t)\theta_0\|^2_{H^{\gamma/2}}\leq c\left[\frac{3}{\kappa}\|f\|_{L^\infty}\right]^{\frac{4\gamma}{\gamma-1}}+  
c\|f\|_{H^{\gamma/2}}^2, \qquad \forall t\geq 1.
$$
In other words, the set
$$
B_{1/2}^\gamma=\left\{\varphi\in H^{\gamma/2}: 
\|\varphi\|_{H^{\gamma/2}}\leq  c\left[2R_\infty\right]^{\frac{4\gamma}{\gamma-1}}+  
c\|f\|_{H^{\gamma/2}}^2\right\}
$$ 
is absorbing for $S_\gamma(t)$, and in particular compact in $L^2$. This concludes the proof of the first part of Theorem \ref{thm:weakglobalattra}, that is, the existence of the global attractor bounded in $H^{\gamma/2}$. Concerning the second
part, note that \eqref{eq:sob3} also provides time integrability of the $H^\gamma$ norm of the solution.
Since $\gamma\in(1,2)$, the inclusion $H^\gamma\subset H^{2-\gamma}$ holds, 
so that time integrability in $H^{2-\gamma}$ follows from the Poincar\'e inequality. At this point, \eqref{eq:sob2} and
the uniform Gronwall lemma yields the existence of an $H^{2-\gamma}$ absorbing set of comparable
size of that in Theorem \ref{thm:H1abs}, on which the restriction of $S_\gamma(t)$ is single-valued. By arguing as
in Section \ref{sub:regattr}, the regularity of the global attractor can therefore be bootstrapped to $H^{2-\gamma/2}$
and the proof of Theorem \ref{thm:weakglobalattra} is achieved.

\section{Uniform absorbing sets}\label{sec:unifabs}
The dependence of the absorbing set with respect to $\gamma$ can certainly be improved. In our case,
by exploiting only the $L^\infty$ maximum principle, we chose a radius such that 
$R_{1,\gamma}\to \infty$ as $\gamma\to 1$. The results of \cite{CCZV15} indicate that, instead, $R_{1,\gamma}$
can be made uniformly bounded for $\gamma\in[1,2)$, as mentioned in the Introduction. By adapting  
the techniques of \cite{CCZV15}, we show in this section how to improve the bounds obtained in 
Theorem \ref{thm:H1abs}. Precisely, we prove the following.
\begin{theorem}\label{thm:H1absbeta}
Let $\gamma_0\in (1,2)$ be arbitrarily fixed. There exists $\beta=\beta(\|f\|_{L^\infty},\gamma_0)\in (0,1/4]$ such that the set
\begin{equation*}
B_1=\left\{\varphi\in H^{2-\gamma}: \|\varphi\|_{H^{2-\gamma}}\leq R_{1}\right\},
\end{equation*}
with 
$$
R_1^2=c\left[2R_\infty\right]^{\frac{4\gamma}{\gamma+\beta-1}}+  
c\|f\|_{H^{2-\gamma}}^2,
$$
is absorbing for $S_\gamma(t)$. Moreover, 
\begin{equation}\label{eq:H1unifbeta}
\sup_{t\geq 0}\sup_{\theta_0\in B_1}\left[\|S_\gamma(t)\theta_0\|^2_{H^{2-\gamma}}+\int_t^{t+1}\|S_\gamma(\tau)\theta_0\|^2_{H^{2-\gamma/2}}\dd \tau\right]\leq 
2 R_1^2.
\end{equation}
\end{theorem}
The improvement compared to Theorem \ref{thm:H1abs} is clear. In this case, the absorbing set has a radius
that is well-behaved (namely, bounded)  as $\gamma\to 1^+$. In other words, it can be be made  \emph{independent} of $\gamma\in (1,2)$,
a crucial step towards the study of the limit $\gamma\to 1^+$. By following the strategy devised in Section \ref{sub:regattr}, we deduce the existence of a regular absorbing set, whose radius is independent of $\gamma$ as well. We state
these considerations in the following corollary.
\begin{corollary}\label{cor:cbdfbdfs}
Let $\gamma_0\in (1,2)$ be arbitrarily fixed. The set
\begin{equation*}
B_2=\left\{\varphi\in H^{2-\gamma/2}: \|\varphi\|_{H^{2-\gamma/2}}
\leq R_{2}\right\}
\end{equation*}
with
$$
R_{2}^2=c\left[2R^2_{1}+ \|f\|^2_{H^{2-\gamma}}\right]\e^{cR^2_{1}},
$$
is absorbing for $S_\gamma(t)$. Moreover, 
\begin{equation}\label{eq:H32unifbis}
\sup_{t\geq 0}\sup_{\theta_0\in B_2}\left[\|S_\gamma(t)\theta_0\|^2_{H^{2-\gamma/2}}+\int_t^{t+1}\|S_\gamma(\tau)\theta_0\|^2_{H^{2}}\dd \tau\right]\leq 
\Q(R_{2}),
\end{equation}
where $\Q(\cdot)$ is a positive increasing function with $\Q(0)=0$.
\end{corollary}
In turn, bounds on the global attractor $A_\gamma$ are uniform, since the global attractor is necessarily contained
in any absorbing set.
\begin{corollary}\label{cor:unifatt}
Let $\gamma_0\in (1,2)$ be arbitrarily fixed, and let $A_\gamma\subset H^{2-\gamma}$ be the global attractor of $S_\gamma(t)$. Then 
\begin{align}\label{eq:unifinal}
\sup_{\gamma \in (1,\gamma_0]} \|A_\gamma\|_{H^{2-\gamma/2}}\leq R_{2}.
\end{align}
Moreover,
\begin{align}
\sup_{\gamma \in (1,\gamma_0]} \dim_{H^{2-\gamma}}A_\gamma<\infty.
\end{align} 
\end{corollary}
The proof of the uniform fractal dimension estimate follows word for word the proof in \cite{CTV14}*{Theorem 6.4} for the critical case,
and it is left to the reader. Clearly, the independence of $\gamma$ of all the estimates translate into the uniformity above.

We devote the rest of the section to the proof of Theorem \ref{thm:H1absbeta}.

\subsection{Absorbing sets H\"older spaces}
In what follows, $\gamma_0\in (1,2)$ is arbitrarily fixed, and we consider 
$$
\gamma\in (1,\gamma_0].
$$
Therefore, we will not worry about the dependence on $\gamma$ of the constant $c_\gamma$ in \eqref{eq:fraclapc}. Indeed, the 
purpose here is to prepare the ground to study the limit $\gamma\to 1^+$, so our assumption is absolutely harmless.
Here, we aim to prove the following result, namely, there exists an absorbing set of H\"older continuous functions.
\begin{proposition}\label{thm:Cbeta}
There exists 
$\beta=\beta(\|f\|_{L^\infty})\in (0,1/4]$ independent of $\gamma$
such that the set
\begin{equation*}
B_{\beta}=\left\{\varphi\in C^\beta\cap H^{2-\gamma}: \|\varphi\|_{C^\beta}\leq c\|f\|_{L^\infty}\right\}
\end{equation*}
is an absorbing set for $S_\gamma(t)$. Moreover, 
\begin{equation}\label{eq:Calphaunifbeta}
\sup_{t\geq 0}\sup_{\theta_0\in B_\beta}\|S(t)\theta_0\|_{C^\beta}\leq 2 c\|f\|_{L^\infty}.
\end{equation}
\end{proposition}
We claim that Theorem \ref{thm:H1abs} follows directly from this result. Indeed, thanks to \cite{CZK15}*{Theorem B.1},
once a uniform control of the type \eqref{eq:Calphaunifbeta} is established,
then the following differential inequality holds
\begin{equation}\label{eq:sobbeta}
\ddt\|\theta\|^2_{H^{2-\gamma}}+\frac14 \|\theta\|^2_{H^{2-\gamma/2}}
\leq c \left[2 c\|f\|_{L^\infty}\right]^{\frac{4\gamma}{\gamma+\beta-1}}+c\|f\|_{H^{2-\gamma}}^2,
\end{equation}
for all $t\geq 0$. The proof of the above inequality is analogous to that of \eqref{eq:sob}, but relies on the strongest a priori control
of a H\"older norm.

\subsection{Finite differences revisited}
We proceed in a similar manner as in Section \ref{sub:sob}, using the same notation for finite differences and operators.
Let $\xi:[0,\infty)\to[0,\infty)$ be a bounded decreasing differentiable function to be determined later. For 
\begin{equation*}
0<\beta \leq \frac14
\end{equation*}
to be fixed later on,
we study the evolution of the quantity $w(x,t;h)$ defined by
\begin{equation}\label{eq:w}
w(x,t;h) =\frac{|\delta_h\theta(x,t)|}{(\xi(t)^2+|h|^2)^{\beta/2}}.
\end{equation}
The main point is that when $\xi(t)=0$ we have that 
$$
\|w(t)\|_{L^\infty_{x,h}}=\esup_{x,h\in\TT} |w(x,t;h)| = \sup_{x\neq y \in\TT} \frac{|\theta(x,t)-\theta(y,t)|}{|x-y|^{\beta}} = [\theta(t)]_{C^\beta}.
$$
From \eqref{eq:findiff} and a short calculation (see~\cite{CZV14}) we obtain that
\begin{align}
L w^2+\frac{D_\gamma[\delta_h\theta] }{(\xi^2+|h|^2)^\beta}
&=2\beta |\dot\xi|\frac{\xi}{\xi^2+|h|^2}w^2 -2\beta \frac{h}{\xi^2+|h|^2}\cdot \delta_h\u \, w^2+
2\frac{\delta_hf}{(\xi^2+|h|^2)^{\beta/2}} w\notag \\
&\leq  2\beta |\dot\xi|\frac{\xi}{\xi^2+|h|^2}w^2 +2\beta \frac{|h|}{\xi^2+|h|^2}|\delta_h\u|w^2+
\frac{4\|f\|_{L^\infty}}{(\xi^2+|h|^2)^{\beta/2}} w\label{eq:ineq1beta}
\end{align}
where $\delta_h\u= \RR^\perp \delta_h\theta$. An analogous of Lemma \ref{lem:nonlinbdd} holds in this case as well.

\begin{lemma}\label{lem:nonlinbddbeta}
There exists a positive constant
$c_2$ such that
\begin{equation}\label{eq:N1beta}
\frac{D_\gamma[\delta_h\theta](x,t)}{(\xi(t)^2+|h|^2)^\beta}
\geq \frac{|w(x,t;h)|^{2+\gamma}}{c_2\|\theta(t)\|^\gamma_{L^\infty}(\xi(t)^2+|h|^2)^{\frac{\gamma(1-\beta)}{2}}} 
\end{equation}
holds for any $x,h\in \TT$ and any $t\geq 0$.
\end{lemma}

\begin{proof}[Proof of Lemma~\ref{lem:nonlinbddbeta}]
Relying on \eqref{eq:dissip1}, namely
\begin{equation}
D_\gamma[\delta_h\theta](x)\geq \frac{c_\gamma}{r^\gamma}|\delta_h\theta(x)|^2-c c_\gamma|\delta_h\theta(x)|\|\theta\|_{L^\infty}\frac{|h|}{r^{1+\gamma}},
\end{equation}
a choice satisfying $r\geq 4(\xi^2+|h|^2)^{1/2}\geq 4|h|$ can be made as
$$
r=\frac{8c \|\theta\|_{L^\infty}}{|\delta_h\theta(x)|} (\xi^2+|h|^2)^{1/2},
$$
from which it follows that
\begin{align*}
D_\gamma[\delta_h\theta](x)&\geq c_\gamma\frac{|\delta_h\theta(x)|^2}{r^\gamma}\left[1-\frac18\frac{|h|}{(\xi^2+|h|^2)^{1/2}}\right]\\
&\geq  c_\gamma\frac{7|\delta_h\theta(x)|^2}{8r^\gamma}\geq\frac{|\delta_h\theta(x)|^{2+\gamma}}{\tilde{c} \|\theta\|^\gamma_{L^\infty} (\xi^2+|h|^2)^{\gamma/2}}
\end{align*}
The lower bound \eqref{eq:N1beta} follows by dividing the above inequality by $(\xi^2+|h|^2)^{\beta}$.
\end{proof}

It is now important to choose the function $\xi$ in a suitable way. To this end, we impose that $\xi$
solves the ordinary differential equation
\begin{equation}\label{eq:ODEbeta}
\dot\xi =- \xi^{1-\frac{2\gamma(1-\beta)}{2+\gamma}},\qquad \xi(0)=1.
\end{equation}
More explicitly,
\begin{equation}\label{eq:xibeta}
\xi(t)=\begin{cases}
\displaystyle \left[1-\frac{2\gamma}{2+\gamma}(1-\beta) t\right]^{\frac{2+\gamma}{2\gamma(1-\beta)}}, \quad &\text{if } t\in [0,t_\beta],\\ \\
0,\quad &\text{if } t \in  (t_\beta,\infty),
\end{cases}
\end{equation}
where 
\begin{equation}\label{eq:regtimebeta}
t_\beta=\frac{2+\gamma}{2\gamma(1-\beta)}.
\end{equation}
We then have the following result.
\begin{lemma}\label{lem:ODEbeta}
Assume that the function $\xi:[0,\infty)\to[0,\infty)$ is given by \eqref{eq:xibeta}. Then
 the estimate
\begin{equation}\label{eq:est1beta}
2\beta |\dot\xi(t)|\frac{\xi(t)}{\xi(t)^2+|h|^2}|w(x,t;h)|^2\leq \frac{|w(x,t;h)|^{2+\gamma}}{8c_2\|\theta(t)\|^\gamma_{L^\infty}(\xi(t)^2+|h|^2)^{\frac{\gamma(1-\beta)}{2}}}  +c\|\theta(t)\|^2_{L^\infty}
\end{equation}
holds pointwise for $x,h \in \TT$ and $t\geq 0$, where $c_2$ is the same constant appearing in \eqref{eq:N1beta}.
\end{lemma}
\begin{proof}[Proof of Lemma~\ref{lem:ODEbeta}]
We again suppress the $t$-dependence in all the estimates below. 
In view of \eqref{eq:ODEbeta} and the fact that $\beta\leq 1/4$, a simple computation shows that
$$
2\beta |\dot\xi|\frac{\xi}{\xi^2+|h|^2}|w(x;h)|^2\leq \frac12\frac{\xi^{2-\frac{2\gamma(1-\beta)}{2+\gamma}}}{\xi^2+|h|^2}|w(x;h)|^2
\leq  \frac12\frac{|w(x;h)|^2}{(\xi^2+|h|^2)^{\frac{\gamma(1-\beta)}{2+\gamma}}}.
$$
Therefore, the $\eps$-Young inequality
\begin{equation}\label{eq:youngbeta}
ab\leq \frac{2\eps}{2+\gamma} a^\frac{2+\gamma}{2}+\frac{\gamma}{(2+\gamma)\eps^{2/\gamma}}b^\frac{2+\gamma}{\gamma},\qquad a,b,\eps>0
\end{equation}
with 
$$
\eps=\frac{\gamma+2}{16c_2\|\theta\|^\gamma_{L^\infty}}
$$ 
and $b=1/2$ implies that
$$
2\beta |\dot\xi|\frac{\xi}{\xi^2+|h|^2}|w(x;h)|^2
\leq\frac{|w(x;h)|^{2+\gamma}}{8c_2\|\theta\|^\gamma_{L^\infty}(\xi^2+|h|^2)^{\frac{\gamma(1-\beta)}{2}}}
+c\|\theta\|^2_{L^\infty},
$$
which is what we claimed.
\end{proof}
In the same fashion, we can estimate the forcing term appearing in \eqref{eq:ineq1beta}.

\begin{lemma}\label{lem:forcebeta}
For every $x,h \in \TT$ and $t\geq 0$ we have
\begin{equation}\label{eq:est2beta}
\frac{4\|f\|_{L^\infty}}{(\xi(t)^2+|h|^2)^{\beta/2}} w(x,t;h)\leq 
\frac{|w(x,t;h)|^{2+\gamma}}{8c_2\|\theta(t)\|^\gamma_{L^\infty}(\xi(t)^2+|h|^2)^{\frac{\gamma(1-\beta)}{2}}}
+ c\|f\|_{L^\infty}^\frac{2+\gamma}{1+\gamma}\|\theta(t)\|^\frac{\gamma}{1+\gamma}_{L^\infty},
\end{equation}
where $c_2$ is the same constant appearing in \eqref{eq:N1beta}.
\end{lemma}

\begin{proof}[Proof of Lemma~\ref{lem:forcebeta}]
Applying the Young inequality
\begin{equation}\label{eq:youngbeta2}
ab\leq \frac{\eps}{2+\gamma} a^{2+\gamma}+\frac{1+\gamma}{(2+\gamma)\eps^{1/(1+\gamma)}}b^\frac{2+\gamma}{1+\gamma},\qquad a,b,\eps>0
\end{equation}
we infer that
$$
\frac{4\|f\|_{L^\infty}}{(\xi^2+|h|^2)^{\beta/2}} w(x;h)\leq 
\frac{|w(x;h)|^{2+\gamma}}{8c_2\|\theta\|^\gamma_{L^\infty}(\xi^2+|h|^2)^{\frac{\gamma(1-\beta)}{2}}}+c (\xi^2+|h|^2)^{\frac{\gamma}{2(1+\gamma)}-\beta}
\|f\|_{L^\infty}^\frac{2+\gamma}{1+\gamma}\|\theta\|^\frac{\gamma}{1+\gamma}_{L^\infty}.
$$
The conclusion follows from the assumption $\beta\leq 1/4$ and the bounds $\xi,|h|\leq 1$.
\end{proof}
If we now apply the bounds \eqref{eq:N1beta}, \eqref{eq:est1beta} and \eqref{eq:est2beta} to \eqref{eq:ineq1beta}, we end up with
\begin{equation}\label{eq:ineq2beta}
\begin{aligned}
L w^2+\frac{1}{2}\frac{ D_\gamma[\delta_h\theta] }{(\xi^2+|h|^2)^\beta}&+\frac{|w|^{2+\gamma}}{4c_2\|\theta\|^\gamma_{L^\infty}(\xi^2+|h|^2)^{\frac{\gamma(1-\beta)}{2}}}\\
&\leq 2\beta \frac{|h|}{\xi^2+|h|^2}|\delta_h\u|w^2+c\left[\|\theta\|^2_{L^\infty}+ \|f\|_{L^\infty}^\frac{2+\gamma}{1+\gamma}\|\theta\|^\frac{\gamma}{1+\gamma}_{L^\infty}\right].
\end{aligned}
\end{equation}

\subsection{Estimates on the nonlinear term}
We would like to stress once more that the only restriction on $\beta$ so far has consisted in
imposing $\beta\in(0,1/4]$. This arose only in the proof of Lemma \ref{lem:forcebeta}.
In order to deal with the Riesz-transform contained in $\delta_h\u$, the H\"older exponent will be
further restricted in terms of the initial datum $\theta_0$ and the forcing term $f$. It is crucial
that this restriction only depends on $\|\theta_0\|_{L^\infty}$ and $\|f\|_{L^\infty}$.
\begin{lemma}\label{lem:rieszbddbeta}
Suppose that $\theta_0\in L^\infty$, and set 
\begin{equation}\label{eq:alphabeta}
\beta=\min\left\{\frac{1}{c_3K_\infty^\frac{3\gamma}{\gamma+2}},\frac14\right\}, \qquad K_\infty=\|\theta_0\|_{L^\infty}+\frac{1}{\kappa}\|f\|_{L^\infty},
\end{equation}
for a universal constant $c_3\geq 64$. Then
\begin{align}
2\beta \frac{|h||\delta_h\u(x,t)|}{\xi(t)^2+|h|^2}|w(x,t;h)|^2&\leq  \frac{1}{2} \frac{D_\gamma[\delta_h\theta](x,t)}{(\xi(t)^2+|h|^2)^\beta}\notag\\
&\quad+ \frac{1}{8c_2K^\gamma_\infty(\xi(t)^2+|h|^2)^{\frac{\gamma(1-\beta)}{2}}} |w(x,t;h)|^{2+\gamma} + c \left(\frac12\right)^\frac{1}{\gamma-1}\label{eq:est3beta},
\end{align}
for every $x,h \in \TT$ and $t\geq 0$, where $c_2$ is the same constant appearing in \eqref{eq:N1beta}.
\end{lemma}

\begin{proof}[Proof of Lemma~\ref{lem:rieszbddbeta}]
Using Lemma \ref{lem:rieszbdd}, for $r\geq 4|h|$ we have the upper bound
$$
|\delta_h\u(x,t)|\leq 
c\left[ r^{\gamma/2} \big(D_\gamma[\delta_h\theta](x,t)\big)^{1/2}+\frac{|h|\|\theta(t)\|_{L^\infty} }{r}\right],
$$
pointwise in $x,h\in \TT$ and $t\geq 0$. Using the Cauchy-Schwarz inequality, we deduce that
\begin{align*}
\frac{2\beta  |h|}{\xi^2+|h|^2}|\delta_h\u(x)||w(x;h)|^2
&\leq \frac{2\beta}{(\xi^2+|h|^2)^{1/2}}|\delta_h\u(x)||w(x;h)|^2\\
&\leq \frac{1}{2} \frac{D_\gamma[\delta_h\theta](x)}{(\xi^2+|h|^2)^\beta} + c\left[\frac{\beta^2}{(\xi^2+|h|^2)^{1-\beta}}r^\gamma |w(x;h)|^4 +\beta \frac{\|\theta\|_{L^\infty}}{r}|w(x;h)|^2\right].
\end{align*}
We then choose $r$ such that
$$
\frac{\beta^2}{(\xi^2+|h|^2)^{1-\beta}}r^\gamma |w(x;h)|^4 =\beta \frac{\|\theta\|_{L^\infty}}{r}|w(x;h)|^2,
$$
namely
$$
r=\frac{\|\theta\|^\frac{1}{1+\gamma}_{L^\infty} (\xi^2+|h|^2)^{\frac{1-\beta}{1+\gamma}}}{\beta^\frac{1}{1+\gamma}w(x;h)^\frac{2}{1+\gamma}}=
\frac{\|\theta\|^\frac{1}{1+\gamma}_{L^\infty}(\xi^2+|h|^2)^\frac{1}{1+\gamma}}{\beta^\frac{1}{1+\gamma}|\delta_h\theta(x)|^\frac{2}{1+\gamma}}.
$$
In view of \eqref{eq:alphabeta}, this is a feasible choice, since $\gamma>1$, $|h|\leq1$ and 
$$
r\geq \frac{\|\theta\|^\frac{1}{1+\gamma}_{L^\infty}}{4^\frac{1}{1+\gamma}\beta^\frac{1}{1+\gamma}\|\theta\|^\frac{2}{1+\gamma}_{L^\infty}}|h|^\frac{2}{1+\gamma}
=\frac{1}{4^\frac{1}{1+\gamma}\beta^\frac{1}{1+\gamma}\|\theta\|^\frac{1}{1+\gamma}_{L^\infty}}|h|^\frac{2}{1+\gamma}\geq 
\frac{1}{4^\frac{1}{1+\gamma}\beta^\frac{1}{1+\gamma}K_\infty^\frac{1}{1+\gamma}}|h|^\frac{2}{1+\gamma}\geq 4|h|.
$$
Thus, thanks to \eqref{eq:alphabeta}, we obtain 
\begin{align*}
2\beta \frac{|h|}{\xi^2+|h|^2}|\delta_h\u(x)||w(x;h)|^2
&\leq \frac{1}{2} \frac{D_\gamma[\delta_h\theta](x)}{(\xi^2+|h|^2)^\beta}
+ c\frac{\beta^\frac{2+\gamma}{1+\gamma}\|\theta\|^\frac{\gamma}{1+\gamma}_{L^\infty} }{(\xi^2+|h|^2)^{\frac{1-\beta}{1+\gamma}}} |w(x;h)|^{2+\frac{2}{1+\gamma}}\\
&\leq \frac{1}{2} \frac{D_\gamma[\delta_h\theta](x)}{(\xi^2+|h|^2)^\beta}
+ c\frac{\beta^\frac{2+\gamma}{1+\gamma}K_\infty^\frac{\gamma}{1+\gamma}}{(\xi^2+|h|^2)^{\frac{1-\beta}{1+\gamma}}} |w(x;h)|^{2+\frac{2}{1+\gamma}}.
\end{align*}
By using the same H\"older exponent as we previously did in \eqref{eq:ineq3}, we have
\begin{align*}
 c\frac{\beta^\frac{2+\gamma}{1+\gamma}K_\infty^\frac{\gamma}{1+\gamma}}{(\xi^2+|h|^2)^{\frac{1-\beta}{1+\gamma}}} |w(x;h)|^{2+\frac{2}{1+\gamma}}
 &\leq \frac{1}{8c_2K^\gamma_\infty(\xi^2+|h|^2)^{\frac{\gamma(1-\beta)}{2}}} |w(x;h)|^{2+\gamma}+ c \beta^\frac{\gamma+2}{\gamma-1} K_\infty^\frac{{3\gamma}}{\gamma-1} (\xi^2+|h|^2)^{1-\beta}\\
 &\leq \frac{1}{8c_2K^\gamma_\infty(\xi^2+|h|^2)^{\frac{\gamma(1-\beta)}{2}}} |w(x;h)|^{2+\gamma}+ c \beta^\frac{\gamma+2}{\gamma-1} K_\infty^\frac{{3\gamma}}{\gamma-1}.
\end{align*}
In view of the restriction on $\beta$ given in \eqref{eq:alphabeta}, we further estimate the last term in the right-hand side above as 
$$
c \beta^\frac{\gamma+2}{\gamma-1} K_\infty^\frac{{3\gamma}}{\gamma-1}\leq  c \left(\frac{1}{2}\right)^\frac{1}{\gamma-1},
$$
so that we can conclude the proof.
\end{proof}

We now proceed with the proof
of H\"older $C^\beta$ estimates, where the exponent $\beta$ is given by \eqref{eq:alphabeta}.

\subsection{Locally uniform H\"older estimates}
From the global bound \eqref{eq:globalLinf}, \eqref{eq:ineq2beta} and the estimate \eqref{eq:est3beta},
it follows that for $\beta$ complying with  \eqref{eq:alphabeta} the function $w^2$ satisfies
\begin{equation*}
L w^2+\frac{|w|^{2+\gamma}}{8c_2K_\infty^\gamma(\xi^2+|h|^2)^{\frac{1-\beta}{2}}}
\leq c\left[K_\infty^2+ \|f\|_{L^\infty}^\frac{2+\gamma}{1+\gamma}K_\infty^\frac{\gamma}{1+\gamma}\right].
\end{equation*}
Taking into account that $\xi^2+|h|^2\leq 1 + {\rm diam}(\TT)^2 = 2$ for all $h \in \TT$, and that $\|f\|_{L^\infty}\leq c_0 K_{\infty}$, we arrive at
\begin{equation}\label{eq:ineq3beta}
L w^2+\frac{|w|^{2+\gamma}}{32c_2K_\infty^\gamma}
\leq cK_\infty^2
\end{equation} 
which holds pointwise for $(x,h) \in \TT \times \TT$.
In the next lemma we show that the above inequality gives uniform control on the $C^\beta$ seminorm
of the solution.

\begin{lemma}\label{lem:Calphabeta}
Assume that $\theta_0\in L^\infty$, and fix $\beta$ as in \eqref{eq:alphabeta}. There exists a time $t_\beta>0$
such that the solution to \eqref{eq:SQGgamma} with initial datum $\theta_0$ is $\beta$-H\"older continuous. Specifically,
\begin{equation*}
[\theta(t)]_{C^\beta}\leq c \left[\|\theta_0\|_{L^\infty}+\frac{1}{\kappa}\|f\|_{L^\infty}\right], \qquad \forall t\geq t_\beta
=\frac{2+\gamma}{2\gamma(1-\beta)}.
\end{equation*}
\end{lemma}

\begin{proof}[Proof of Lemma~\ref{lem:Calphabeta}]
Thanks to \eqref{eq:ineq3beta}, the function
$$
\psi(t)=\|w(t)\|_{L^\infty_{x,h}}^2
$$
satisfies the differential inequality
\begin{equation}\label{eq:diffineqbeta}
\frac{\dd}{\dd t} \psi+\frac{1}{32c_2K^\gamma_\infty} \psi^{1+\gamma/2}\leq cK_\infty^2.
\end{equation}
From the definition of $w$,
$$
\psi(0)\leq\frac{4\|\theta_0\|^2_{L^\infty}}{\xi_0^{2\beta}}=4\|\theta_0\|^2_{L^\infty}\leq 4K_\infty^2.
$$
From a standard comparison for ODEs it immediately follows that 
\begin{equation}\label{eq:bddpsibeta}
\psi(t)\leq cK^2_\infty, \qquad \forall t\geq 0,
\end{equation}
for some sufficiently large constant $c>0$.
With \eqref{eq:bddpsibeta} at hand, we have thus proven that 
$$
[\theta(t)]_{C^\beta}^2=\psi(t) \leq c K_\infty^2, \qquad \forall t\geq t_\beta,
$$
where $t_\beta$ is given by \eqref{eq:regtimebeta}, thereby concluding the proof.
\end{proof}
It is now clear that the proof of Proposition \ref{thm:Cbeta} is now achieved. Indeed,
Lemma \ref{lem:Calphabeta} provides a uniform (with respect to the initial datum in the
$L^\infty$ absorbing set) estimate of the H\"older seminorm, with an associated regularization
time that, ultimately, depends only on $\|f\|_{L^\infty}$ through the H\"older exponent $\beta$.

\section{Upper semicontinuity of the attractors}\label{sec:stable}
To conclude the article, it remains to prove the second part of Theorem \ref{thm:stability}.
Fix $\gamma_0\in (1,2)$ and $\gamma\in (1,\gamma_0]$.
To establish \eqref{eq:stability}, we preliminary note the following. 
Let
$$
\mathcal{U}:=\bigcup_{\gamma\in (1,\gamma_0]} A_\gamma.
$$
Calling $S_1(t):H^1\to H^1$ the semigroup generated by the critical SQG equation ($\gamma=1$), 
by the triangle inequality and the invariance properties of the global attractor, we have
\begin{align}
\dist_{H^1}(A_\gamma,A_1)&\leq \dist_{H^1}(A_\gamma, S(t)A_\gamma)+\dist_{H^1}(S(t)A_\gamma,A_1)\notag\\
&\leq \dist_{H^1}(A_\gamma, S_1(t)A_\gamma)+\dist_{H^1}(S_1(t)\mathcal{U},A_1)\notag\\
&= \dist_{H^1}(S_\gamma(t)A_\gamma, S_1(t)A_\gamma)+\dist_{H^1}(S_1(t)\mathcal{U},A_1)\label{eq:conver1},
\end{align}
for every $t>0$.
The proof of Theorem \ref{thm:stability} essentially relies on two main ingredients. The first is the uniform bound
proven in Corollary \ref{cor:unifatt} of the global attractors $A_\gamma$. This, together with the Poincar\'e inequality, 
ensures that $\mathcal{U}$
is bounded in $H^{2-\gamma_0/2}$. In turn, the second term in \eqref{eq:conver1} vanishes as $t\to\infty$, 
due to the fact that  $A_1$ attracts bounded sets of $H^1$ (and hence, in particular, bounded in $H^{2-\gamma_0/2}$), namely
\begin{align}
\lim_{t\to\infty}\dist_{H^1}(S_1(t)\mathcal{U},A_1)=0.
\end{align}
In order to deal with the first term in \eqref{eq:conver1}, we need a proper convergence estimate
(for any fixed time) of the solution of the subcritical SQG equation to that of the critical one. 
\begin{proposition}\label{prop:conver}
Let $\gamma\in [1, 3/2]$, and assume that $\theta_0\in A_\gamma$. Then
\begin{align}
\|S_\gamma(t)\theta_0-S_1(t)\theta_0\|_{H^1}\leq c(\gamma-1), \qquad \forall t\geq0,
\end{align}
where $c>0$ depends only on the uniform bound \eqref{eq:unifinal}, and is in particular \emph{independent} of $\gamma$.

\end{proposition}

The proof is carried over in the next section. We here mention a straightforward, yet crucial, consequence of the above proposition.
\begin{corollary}\label{cor:conv}
For every $t\geq 0$, there holds
\begin{equation}\label{eq:conv3}
\lim_{\gamma\to1^+}\dist_{H^1}(S_{\gamma}(t)A_\gamma, S_1(t)A_\gamma)=0.
\end{equation}
\end{corollary}
\begin{proof}
If not, there exists $t\geq 0$, $\eps>0$ and sequences $\gamma_n\in(1,\gamma_0]$ with $\gamma_n\to 1^+$ as $n\to\infty$,
$\theta_{0,n}\in A_{\gamma_n}$ such that,
$$
\inf_{\theta_0\in A_{\gamma_n}} \|S_{\gamma_n}(t)\theta_{0,n}-S_1(t)\theta_0\|_{H^1}\geq \eps, \qquad \forall n\in\N.
$$
In particular,  we necessarily have that
$$
\|S_{\gamma_n}(t)\theta_{0,n}-S_1(t)\theta_{0,n}\|_{H^1}\geq \eps, \qquad \forall n\in\N.
$$
However, Propostion \ref{prop:conver} yields the contradiction
$$
\eps\leq \|S_{\gamma_n}(t)\theta_{0,n}-S_1(t)\theta_{0,n}\|_{H^1}\leq (\gamma_n-1)\to 0, \qquad \text{as } n\to \infty.
$$
The proof is achieved.
\end{proof}

The stability
property \eqref{eq:stability} follows immediately.
Let $\eps>0$ be arbitrarily fixed. Since $A_1$ is the attractor of $S_1(t)$, there exists
$t_\eps>0$ such that
$$
\dist_{H^1}(S(t_\eps)\mathcal{U},A_1)\leq \frac{\eps}{2}.
$$
Now, Corollary \ref{cor:conv} ensures the existence of some $\gamma_\eps\in (1,\gamma_0]$ such that
$$
\dist_{H^1}(S_\gamma(t_\eps)A_\gamma, S_1(t_\eps)A_\gamma)\leq \frac{\eps}{2}, \qquad \forall \gamma\in(1, \gamma_\eps).
$$
Thus, recalling \eqref{eq:conver1} we have
$$
\dist_{H^1}(A_\gamma,A_1)\leq \dist_{H^1}(S_\gamma(t_\eps)A_\gamma, S_1(t_\eps)A_\gamma)
+\dist_{H^1}(S_1(t_\eps)\mathcal{U},A_1)\leq \eps, \qquad \forall \gamma\in(1, \gamma_\eps).
$$
Since $\eps$ was arbitrary to begin with, we can conclude that
$$
\lim_{\gamma\to1^+}\dist_{H^1}(A_\gamma,A_1)=0,
$$
as wanted.

\subsection{Convergence estimates}
We will need the following general result on the fractional laplacian, to properly estimate the difference between 
solutions of the subcritical and the critical SQG equations stated in Proposition \ref{prop:conver}.
\begin{lemma}\label{lem:gamma}
Let $m\geq 0$, $\gamma>1$ and $s\geq\gamma-1$. Then
\begin{equation}\label{eq:gamma}
\|(\Lambda^{\gamma-1}-I)\varphi\|_{H^m} \leq \frac{\gamma-1}{s}\|\varphi\|_{H^{m+s}}, \qquad \forall \varphi \in H^{m+s}.
\end{equation}
\end{lemma}

\begin{proof}
Using Fourier series, we have
\begin{align*}
\|(\Lambda^{\gamma-1}-I)\varphi\|_{H^m}^2&=\sum_{k \in \ZZ^2_*} |k|^{2m} (|k|^{\gamma-1}-1)^2|\hat\varphi_k|^2=\sum_{k \in \ZZ^2_*} |k|^{2m+2s}\left( \frac{|k|^{\gamma-1}-1}{|k|^s}\right)^2|\hat\varphi_k|^2\\
&\leq \left(\frac{\gamma-1}{s}\right)^2\sum_{k \in \ZZ^2_*} |k|^{2m+2s}|\hat\varphi_k|^2=\left(\frac{\gamma-1}{s}\right)^2\|\varphi\|_{H^{m+s}}^2.
\end{align*}
Above, we used the fact that the real function 
$$
f(x)=\frac{x^{\gamma-1}-1}{x^s},\qquad x\geq 1,
$$
is bounded by 1 when $s=\gamma-1$, and has a maximum at a point $\bar x\geq 1$ such that
$$
f(\bar x)=\frac{\gamma-1}{s} \frac{1}{\bar x^{\gamma + s - 1}}
$$
when $s>\gamma-1$.
\end{proof}

We therefore compare the evolution of the critical SQG equation
\begin{equation}\label{eq:critSQG} 
\begin{cases}
\de_t\theta +\u \cdot \nabla \theta+\Lambda\theta=f,\\
\theta(0)=\theta_0,\qquad\u=\nabla^\perp\Lambda^{-1}\theta,
\end{cases}
\end{equation}
and the subcritical one
\begin{equation}\label{eq:subSQG}
\begin{cases}
\de_t\theta^\gamma +\u^\gamma \cdot \nabla\theta^\gamma +\Lambda^\gamma\theta^\gamma=f,\\
\theta^\gamma(0)=\theta_0,\qquad\u^\gamma=\nabla^\perp\Lambda^{-1}\theta^\gamma,
\end{cases}
\end{equation}
with the same initial datum $\theta_0\in A_\gamma \subset H^{2-\gamma/2} \subset H^{2-\gamma_0/2}$.
The difference 
\begin{equation}
\eta=\theta-\theta^\gamma, \qquad \w= \u-\u^\gamma
\end{equation}
is easily seen to solve
\begin{equation}\label{eq:diff}
\begin{cases}
\de_t\eta +\u \cdot \nabla\eta+\w\cdot \nabla \theta^\gamma +\Lambda^\gamma\eta=(\Lambda^\gamma-\Lambda)\theta,\\
\eta(0)=0.
\end{cases}
\end{equation}
From \eqref{eq:diff}, we deduce that
\begin{align}\label{eq:h1estma}
\frac12\ddt \|\eta\|_{H^1}^2+ \|\eta\|^2_{H^{1+\gamma/2}}=
&\int_{\TT}( \u \cdot\nabla\eta) \Delta \eta\, \dd x
+\int_{\TT}(\w\cdot\nabla \theta^\gamma )\Delta\eta\, \dd x-\int_{\TT}(\Lambda^\gamma-\Lambda)\theta  \Delta \eta\, \dd x.
\end{align}
We now bound the right-hand side term by term, by making use of the Sobolev embedding \eqref{eq:imbe} in the particular cases
\begin{align}
\|\varphi\|_{L^{4/\gamma}}\leq c\| \Lambda^{1-\gamma/2}\varphi\|_{L^2},\qquad
\|\varphi\|_{L^\frac{4}{2-\gamma}}\leq c\| \Lambda^{\gamma/2}\varphi\|_{L^2}.
\end{align}
Using integration by parts and the incompressibility of $\u$, the boundedness of the Riesz transform and the Poincar\'e inequality,  
the first integral can be estimated as
\begin{align}
\left|\int_{\TT}( \u \cdot\nabla\eta) \Delta \eta\, \dd x\right| 
&\leq \|\nabla\u\|_{L^{4/\gamma}} \|\nabla \eta\|_{L^2} \|\nabla\eta\|_{L^{\frac{4}{2-\gamma}}} \notag\\
&\leq c\|\theta\|_{H^{2-\gamma/2}} \|\eta\|_{H^{1}} \|\eta\|_{H^{1+\gamma/2}} \notag\\
&\leq c\|\theta\|_{H^{3/2}} \|\eta\|_{H^{1}} \|\eta\|_{H^{1+\gamma/2}}\notag \\
&\leq c\|\theta\|_{H^{3/2}}^2 \|\eta\|_{H^{1}}^2+\frac18 \|\eta\|_{H^{1+\gamma/2}}^2\label{eq:raw1}.
\end{align}
Concerning the second term, after integration by parts, we have two contributions, namely
\begin{align}
\left|\int_{\TT}(\w\cdot\nabla \theta^\gamma )\Delta\eta\, \dd x\right| \leq \int_{\TT}|\nabla\w||\nabla \theta^\gamma||\nabla\eta|\, \dd x+\int_{\TT}|\w||\Delta \theta^\gamma||\nabla\eta|\, \dd x.
\end{align}
Since
\begin{align}
\int_{\TT}|\nabla\w||\nabla \theta^\gamma||\nabla\eta|\, \dd x
&\leq\|\nabla\w\|_{L^2}\|\nabla \theta^\gamma\|_{L^{4/\gamma}}\|\nabla\eta\|_{L^\frac{4}{2-\gamma}}\\
&\leq c \|\theta^\gamma\|_{H^{2-\gamma/2}}\|\eta\|_{H^1} \|\nabla\eta\|_{H^{1+\gamma/2}},
\end{align}
and
\begin{align}
\int_{\TT}|\w||\Delta \theta^\gamma||\nabla\eta|\, \dd x
&\leq \|\w\|_{L^{4/\gamma}} \|\Delta \theta^\gamma\|_{L^2}\|\nabla\eta\|_{L^\frac{4}{2-\gamma}}\\
&\leq  \| \theta^\gamma\|_{H^2}\|\eta\|_{H^{1-\gamma^2}}\|\eta\|_{H^{1+\gamma/2}}\\
&\leq  \| \theta^\gamma\|_{H^2}\|\eta\|_{H^{1}}\|\eta\|_{H^{1+\gamma/2}},
\end{align}
a simple use of the Young inequality leads us to
\begin{align}\label{eq:raw2}
\left|\int_{\TT}(\w\cdot\nabla \theta^\gamma )\Delta\eta\, \dd x\right|
\leq c\left[\|\theta^\gamma\|_{H^{2-\gamma/2}}^2+ \|\theta^\gamma\|_{H^2}^2\right] \|\eta\|_{H^1}^2 + \frac14 \|\eta\|_{H^{1+\gamma/2}}^2
\end{align}
It remains to bound the last term in \eqref{eq:h1estma}. Once again, integration by parts and standard inequalities entail
\begin{align}
\left|\int_{\TT}(\Lambda^\gamma-\Lambda)\theta  \Delta \eta\, \dd x\right|&=
\left|\int_{\TT}(\Lambda^{\gamma-1}-I)\Lambda^{2-\gamma/2}\theta \, \Lambda^{1+\gamma/2} \eta\, \dd x\right|\notag\\
&\leq \| (\Lambda^{\gamma-1}-I)\theta\|_{H^{2-\gamma/2}}\|\eta\|_{H^{1+\gamma/2}}\notag\\
&\leq c\| (\Lambda^{\gamma-1}-I)\theta\|_{H^{2-\gamma/2}}^2+\frac14\|\eta\|_{H^{1+\gamma/2}}^2\label{eq:raw3}.
\end{align}
Collecting \eqref{eq:raw1}, \eqref{eq:raw2} and \eqref{eq:raw3}, and plugging the result into \eqref{eq:h1estma}, we arrive at
\begin{align}\label{eq:h1estma2}
\ddt \|\eta\|_{H^1}^2+ \|\eta\|^2_{H^{1+\gamma/2}}\leq
c\left[ \|\theta\|_{H^{3/2}}^2+\|\theta^\gamma\|_{H^{2-\gamma/2}}^2+ \|\theta^\gamma\|_{H^2}^2\right] \|\eta\|_{H^1}^2+
c\| (\Lambda^{\gamma-1}-I)\theta\|_{H^{2-\gamma/2}}^2.
\end{align}
Finally, we make use of \eqref{eq:gamma} with the choice
$$
m=2-\frac\gamma2, \qquad  s=\frac12,
$$ 
which complies with the assumptions of Lemma \ref{lem:gamma} since $\gamma\leq 3/2$, so that from \eqref{eq:h1estma2} we obtain
\begin{align}\label{eq:h1estma3}
\ddt \|\eta\|_{H^1}^2+ \|\eta\|^2_{H^{1+\gamma/2}}\leq
c\left[ \|\theta\|_{H^{3/2}}^2+\|\theta^\gamma\|_{H^{2-\gamma/2}}^2+ \|\theta^\gamma\|_{H^2}^2\right] \|\eta\|_{H^1}^2+
c(\gamma-1)^2\| \theta\|_{H^{5/2-\gamma/2}}^2.
\end{align}
Note that from the results in \cites{CTV14,CCZV15}, if $\theta_0 \in H^{5/2-\gamma/2}\subset H^{2-\gamma_0/2}$, then 
\begin{equation}
\sup_{t\geq 0}\sup_{\theta_0\in A_\gamma}\left[\|S_1(t)\theta_0\|^2_{H^{2-\gamma_0/2}}+\int_t^{t+1}\|S_\gamma(\tau)\theta_0\|^2_{H^{5/2-\gamma_0/2}}\dd \tau\right]\leq c,
\end{equation}
where $c$ only dependence on the (uniform) bounds on $A_\gamma$. Furthermore, thanks to Corollary \ref{cor:cbdfbdfs} and the above
estimate, all the quantities in \eqref{eq:h1estma3} are integrable in time.
Therefore, the Gronwall lemma, together with the fact that $\eta(0)=0$, yields
\begin{align}
\|\eta(t)\|_{H^1}^2\leq c(\gamma-1)^2, \qquad \forall t\geq0,
\end{align}
which is precisely the claim of Proposition \ref{prop:conver}.

\section*{Acknowledgements}
The author would like to thank Jacob Bedrossian, Alexey Cheskidov, 
Peter Constantin, Hao Jia and Vlad Vicol for helpful discussions.
This work was supported in part by an AMS-Simons Travel Award.




\begin{bibdiv}
\begin{biblist}

\bib{B02}{article}{
   author={Berselli, Luigi C.},
   title={Vanishing viscosity limit and long-time behavior for 2D
   quasi-geostrophic equations},
   journal={Indiana Univ. Math. J.},
   volume={51},
   date={2002},
   pages={905--930},
}

\bib{CV10a}{article}{
   author={Caffarelli, Luis A.},
   author={Vasseur, Alexis},
   title={Drift diffusion equations with fractional diffusion and the
   quasi-geostrophic equation},
   journal={Ann. of Math. (2)},
   volume={171},
   date={2010},
   pages={1903--1930},
}


\bib{CD14}{article}{
   author={Cheskidov, A.},
   author={Dai, M.},
   title={The existence of a global attractor for the forced critical surface quasi-geostrophic equation in $L^2$},
   journal = {ArXiv e-prints},
   eprint = {1402.4801},
   date = {2014},
}

\bib{CD15}{article}{
   author={Cheskidov, A.},
   author={Dai, M.},
   title={Determining modes for the surface Quasi-Geostrophic equation},
   journal = {ArXiv e-prints},
   eprint = {1507.01075},
   date = {2015},
}

\bib{CCZV15}{article}{
   author={Constantin, Peter},
   author={Coti Zelati, M.},
   author={Vicol, Vlad},
   title={Uniformly attracting limit sets for the critically dissipative SQG equation},
   journal={Nonlinearity},
   volume={29},
   date={2016},
   number={2},
   pages={298--318},
}

\bib{CMT94}{article}{
   author={Constantin, Peter},
   author={Majda, Andrew J.},
   author={Tabak, Esteban},
   title={Formation of strong fronts in the $2$-D quasigeostrophic thermal
   active scalar},
   journal={Nonlinearity},
   volume={7},
   date={1994},
   pages={1495--1533},
}

\bib{CTV14}{article}{
   author={Constantin, Peter},
   author={Tarfulea, Andrei},
   author={Vicol, Vlad},
   title={Absence of anomalous dissipation of energy in forced two
   dimensional fluid equations},
   journal={Arch. Ration. Mech. Anal.},
   date={2014},
   number={3},
   pages={875--903},
}

\bib{CTV13}{article}{
   author={Constantin, Peter},
   author={Tarfulea, Andrei},
   author={Vicol, Vlad},
   title = {Long time dynamics of forced critical SQG},
   journal={Comm. Math. Phys.},
   volume={335},
   date = {2014},
   pages={93--141},
}

\bib{CV12}{article}{
   author={Constantin, Peter},
   author={Vicol, Vlad},
   title={Nonlinear maximum principles for dissipative linear nonlocal
   operators and applications},
   journal={Geom. Funct. Anal.},
   volume={22},
   date={2012},
   pages={1289--1321},
}

\bib{CW99}{article}{
   author={Constantin, Peter},
   author={Wu, Jiahong},
   title={Behavior of solutions of 2D quasi-geostrophic equations},
   journal={SIAM J. Math. Anal.},
   volume={30},
   date={1999},
   pages={937--948},
}

\bib{CW08}{article}{
    author={Constantin, Peter},
    author={Wu, Jiahong},
    title={Regularity of {H}\"older continuous solutions of the supercritical quasi-geostrophic equation},
    journal={Ann. Inst. H. Poincar\'e Anal. Non Lin\'eaire},
    volume={25},
    date={2008},
    pages={1103--1110},
}

\bib{CC04}{article}{
   author={C{\'o}rdoba, Antonio},
   author={C{\'o}rdoba, Diego},
   title={A maximum principle applied to quasi-geostrophic equations},
   journal={Comm. Math. Phys.},
   volume={249},
   date={2004},
   pages={511--528},
}

\bib{CZ13}{article}{
   author={Coti Zelati, Michele},
   title={On the theory of global attractors and Lyapunov functionals},
   journal={Set-Valued Var. Anal.},
   volume={21},
   date={2013},
   pages={127--149},
}

\bib{CZV14}{article}{
   author={Coti Zelati, M.},
   author={Vicol, Vlad},
   title={On the global regularity for the supercritical SQG equation},
   journal={Indiana Univ. Math. J.},
   volume={65},
   date={2016},
   number={2},
   pages={535--552},
}

\bib{CZK15}{article}{
   author={Coti Zelati, M.},
   author={Kalita, P.},
   title={Smooth attractors for weak solutions of the SQG equation with
   critical dissipation},
   journal={Discrete Contin. Dyn. Syst. Ser. B},
   volume={22},
   date={2017},
   number={5},
   pages={1857--1873},
}

\bib{Dong10}{article}{
	author = {Dong, Hongjie},
	title = {Dissipative quasi-geostrophic equations in critical Sobolev spaces: smoothing effect and global well-posedness},
	journal = {Discrete Contin. Dyn. Syst.},
	volume = {26},
	date = {2010},
	number = {4},
	pages = {1197--1211},
}

\bib{DP09}{article}{
	author={Dong, Hongjie},
	author={Pavlovic, Natasa},
	title={Regularity Criteria for the Dissipative Quasi-Geostrophic Equations in H\"older Spaces},
	journal={Commun. Math. Phys.},
	volume={290},
	date={2009},
	pages={801--812},
}

\bib{FPV09}{article}{
   author={Friedlander, S.},
   author={Pavlovi{\'c}, N.},
   author={Vicol, V.},
   title={Nonlinear instability for the critically dissipative
   quasi-geostrophic equation},
   journal={Comm. Math. Phys.},
   volume={292},
   date={2009},
   pages={797--810},
}


\bib{Hale}{book}{
   author={Hale, J.K.},
   title={Asymptotic behavior of dissipative systems},
   publisher={American Mathematical Society},
   place={Providence, RI},
   date={1988},
}

\bib{Ju04}{article}{
   author={Ju, Ning},
   title={Existence and uniqueness of the solution to the dissipative 2D
   quasi-geostrophic equations in the Sobolev space},
   journal={Comm. Math. Phys.},
   volume={251},
   date={2004},
   pages={365--376},
}

\bib{J05}{article}{
   author={Ju, Ning},
   title={The maximum principle and the global attractor for the dissipative
   2D quasi-geostrophic equations},
   journal={Comm. Math. Phys.},
   volume={255},
   date={2005},
   pages={161--181},
}

\bib{Ju07}{article}{
   author={Ju, Ning},
   title={Dissipative 2D quasi-geostrophic equation: local well-posedness,
   global regularity and similarity solutions},
   journal={Indiana Univ. Math. J.},
   volume={56},
   date={2007},
   pages={187--206},
}

\bib{KP88}{article}{
   author={Kato, Tosio},
   author={Ponce, Gustavo},
   title={Commutator estimates and the Euler and Navier-Stokes equations},
   journal={Comm. Pure Appl. Math.},
   volume={41},
   date={1988},
   pages={891--907},
}

\bib{KPV91}{article}{
   author={Kenig, Carlos E.},
   author={Ponce, Gustavo},
   author={Vega, Luis},
   title={Well-posedness of the initial value problem for the Korteweg-de
   Vries equation},
   journal={J. Amer. Math. Soc.},
   volume={4},
   date={1991},
   pages={323--347},
}

\bib{KN09}{article}{
   author={Kiselev, A.},
   author={Nazarov, F.},
   title={A variation on a theme of Caffarelli and Vasseur},
   journal={Zap. Nauchn. Sem. S.-Peterburg. Otdel. Mat. Inst. Steklov. (POMI)},
   volume={370},
   date={2009},
   pages={58--72, 220},
}

\bib{KNV07}{article}{
   author={Kiselev, A.},
   author={Nazarov, F.},
   author={Volberg, A.},
   title={Global well-posedness for the critical 2D dissipative quasi-geostrophic equation},
   journal={Invent. Math.},
   volume={167},
   date={2007},
   pages={445--453},
}

\bib{MV98}{article}{
   author={Melnik, Valery S.},
   author={Valero, Jos{\'e}},
   title={On attractors of multivalued semi-flows and differential
   inclusions},
   journal={Set-Valued Anal.},
   volume={6},
   date={1998},
   pages={83--111},
}

\bib{P87}{book}{
   author={Pedlosky, Joseph},
   title={Geophysical Fluid Dynamics},
   publisher={Springer},
   place={New York},
   date={1987},
}

\bib{R95}{book}{
   author={Resnick, Serge G.},
   title={Dynamical problems in non-linear advective partial differential
   equations},
   note={Thesis (Ph.D.)--The University of Chicago},
   date={1995},
}

\bib{SellYou}{book}{
   author={Sell, G.R.},
   author={You, Y.},
   title={Dynamics of evolutionary equations},
   publisher={Springer-Verlag, New York},
   date={2002},
}

\bib{Sil10a}{article}{
   author={Silvestre, Luis},
   title={Eventual regularization for the slightly supercritical
   quasi-geostrophic equation},
   journal={Ann. Inst. H. Poincar\'e Anal. Non Lin\'eaire},
   volume={27},
   date={2010},
   pages={693--704},
}

\bib{T3}{book}{
   author={Temam, R.},
   title={Infinite-dimensional dynamical systems in mechanics and physics},
   publisher={Springer-Verlag},
   place={New York},
   date={1997},
}

\bib{WT13}{article}{
   author={Wang, M.},
   author={Tang, Y.},
   title={On dimension of the global attractor for 2D quasi-geostrophic
   equations},
   journal={Nonlinear Anal. Real World Appl.},
   volume={14},
   date={2013},
   pages={1887--1895},
}

\end{biblist}
\end{bibdiv}
\end{document}